\documentclass[12pt,reqno,a4paper]{amsart}
\usepackage{amsmath,mathrsfs}
\usepackage{amssymb,amsfonts}
\usepackage{bbm}

\usepackage{color}
\usepackage{mathtools}
\usepackage{todonotes}
\usepackage[english]{cleveref}

\newcommand\NoBlackBoxes{\global\overfullrule0pt}
\NoBlackBoxes

\theoremstyle{plain}\newtheorem{thm}{Theorem}[section]
\theoremstyle{plain}\newtheorem{cor}[thm]{Corollary}
\theoremstyle{definition}\newtheorem{rem}[thm]{Remark}
\theoremstyle{definition}
\theoremstyle{plain}
\theoremstyle{plain}\newtheorem{lem}[thm]{Lemma}
\theoremstyle{plain}\newtheorem{prop}[thm]{Proposition}
\theoremstyle{definition}
\theoremstyle{plain}

\numberwithin{equation}{section}

\newcommand{\E}{{\mathbb{E}}}

\newcommand{\R}{{\mathbb{R}}}

\renewcommand{\P}{{\mathbb{P}}}

\newcommand{\WK}[1]{\mathbb{P}\left(#1\right)}
\newcommand{\EW}[1]{\mathbb{E}\left[#1\right]}
\newcommand{\V}[1]{\mathbb{V}\left(#1\right)}

\newcommand{\1}{\mathbbm{1}}

\DeclareSymbolFont{Symbols}{OMS}{cmsy}{m}{n}
\SetSymbolFont{Symbols}{bold}{OMS}{cmsy}{b}{n}
\DeclareMathSymbol{\Setminus}{\mathbin}{Symbols}{"6E}

\let\eps\varepsilon

\newcommand{\nconv}{\xrightarrow[]{N\to\infty}}
\newcommand{\Pconv}{\xrightarrow[N\to\infty]{\P}}

\newcommand{\dconv}{\xRightarrow{N\to\infty}}

\crefformat{lem}{Lemma #2#1#3}
\crefformat{thm}{Theorem #2#1#3}
\crefformat{conj}{Conjecture #2#1#3}
\crefformat{conprobs}{Conjectures and Open Problems#2#1#3}
\crefformat{cor}{Corollary #2#1#3}
\crefformat{prop}{Proposition #2#1#3}
\crefformat{defn}{Definition #2#1#3}
\crefformat{rem}{Remark #2#1#3}
\crefformat{bsp}{Beispiel #2#1#3}
\crefformat{ex}{Example #2#1#3}
\crefformat{notation}{Notation #2#1#3}
\crefformat{equation}{(#2#1#3)}
\crefformat{align}{(#2#1#3)}
\crefformat{chapter}{Chapter #2#1#3}
\crefformat{section}{Section #2#1#3}
\crefformat{appendix}{Appendix #2#1#3}
\crefformat{subsection}{Section #2#1#3}
\crefformat{subsubsection}{Section #2#1#3}
\crefformat{figure}{Figure #2#1#3}
\crefformat{algorithm}{Algorithmus #2#1#3}

\allowdisplaybreaks

\begin{document}

\title[Spectral properties of the SBM and hitting times of random walks]{Spectral properties of the stochastic block model and their application to hitting times of random walks}

\author[Matthias L\"owe]{Matthias L\"owe }

\email[Matthias L\"owe]{maloewe@uni-muenster.de}

\author[Sara Terveer]{Sara Terveer}

\email[Sara Terveer]{terveer@math.lmu.de}

\begin{abstract}
We analyze hitting times of simple random walk on realizations of the stochastic block model. We show that under some
natural assumptions the hitting time averaged over the target vertex asymptotically almost surely 
given by $N(1+o(1))$. On the other hand, the hitting time averaged over 
the starting vertex asymptotically almost surely depends on expected degrees in the block the target vertex is in.

We also show a central limit theorem for the hitting time averaged over the starting vertex. Our main techniques are 
a spectral decomposition of these hitting times, a spectral analysis of the adjacency matrix and the graph Laplacian.
\end{abstract}


\maketitle


\section{Introduction}\label{sec:intro}
Random walks on random graphs are very topical in contemporary probability theory, see e.g.\ \cite{AGHHN22, Sylvester21, Frieze_et_al18,
BLPS18} and the references therein. Such random walks can be considered a special instance of random walk in random environment and the principal question is, if and how
this random environment, hence the random graph, influences characteristic properties of such a random walk. For example, \cite{FR08} and \cite{AGHHN22} analyze the mixing 
time of simple random walks on random graphs, while the focus in \cite{sood}, \cite{LoeweTorres}, \cite{LT2020}, and \cite{LT2022} is on the (average) hitting times. 
Note however, that most results on random walks on random graphs consider (sparse or dense) Erd\H{o}s-R\'enyi graphs (the note 
\cite{HL19} is a slight exception, there 
the authors consider random (Erd\H{o}s-R\'enyi type) hypergraph models). As in \cite{LoeweTorres}, \cite{LT2020}, and \cite{LT2022} in this note we will consider
average hitting times of random walks on random graphs. Our main technique will again be a spectral representation of the hitting 
times. Note that in \cite{ottolini2023concentration} the authors take a 
different approach to hitting times for very dense Erd\H{o}s-R\'enyi random graphs (i.e.\ for them $p$ is constant) exploiting the fact 
that for such graphs the diameter is $2$ with high probability. 

In the present situation, the random environment, however, will be given by a stochastic block model.  
To be more concrete, we consider a random graph $G_N(M,P)=(V_N,E_N)$. $G_N(M,P)$ will be constructed according to the stochastic block model, i.e. the vertex set 
$V_N=\{1,\dots,N\}$ is fixed and consists of $M<N$ disjoint ``blocks'' $V^{(1)}_N,\dots,V^{(M)}_N$ with $N_1,\dots,N_M$ vertices, where $\sum\limits_{m=1}^MN_m=N$. Here, we will consider the situation where $N$ is a multiple of $M$ and choose $N_1=\dots=N_M=\frac{N}{M}$.

The unoriented edges of the graph are realized by stochastically independent random variables $\eps_{v,w}\in\{0,1\}$ 
(for the edge $\{v,w\} ,v, w\in V_N$). The probabilities 
for $\eps_{v,w}=1$ are given by a (symmetric) $M\times M$ matrix denoted by $P\coloneqq P_M$. Hence the probabilities for $\eps_{v,w}=1$ just depend on the blocks the vertices are 
in, and the entries $p_{i,j}(N)$ denote the probability for an edge to 
exist between two vertices $v\in V^{(i)}_N$ and $w\in V^{(j)}_N$.
Notice that we allow loops (denoted by multisets $\{v,v\}$, $v\in V$) in the graph for technical purposes of the graph's adjacency matrix, however, their removal should not influence the results. We will denote the set of loops by $L=L_N$.

In the present note, we assume all edge probabilities between different blocks to be identical $q:=p_0(N)$, i.e.\ all off-diagonal entries of $P_M$ are $q$. 
The diagonal entries $p_{i,i}(N)$ are abbreviated by $p_i(N)$. 
We assume that the $p_i(N)$ are sorted in descending order, i.e. $p_1(N)\geq p_2(N)\geq\dots \geq p_M(N)$. 
This can easily be obtained by rearranging the blocks of the model in the corresponding order.
We denote the block number of a vertex $v\in V$ by $B(v)$, i.e. for $m=1,\dots,M$, if $v\in V^{m}$, then $B(v)=m$.

We will additionally impose the relatively mild connectivity conditions. To formulate them we refer to the definition \eqref{eq:gamma_m}of the expected degree of a vertex in block $m$, $\gamma_m$,  and set
$$\gamma_{\min}\coloneqq \min\limits_{m=1,\dots,M}\gamma_m.$$
as well as
$$\gamma_{\max}\coloneqq \max\limits_{m=1,\dots,M}\gamma_m.$$
Then our first condition reads
	\begin{equation}
\frac{M \log^4(N)}{N p_m(N) +N(M-1)q(N)}\cdot\frac{\gamma_{\max}^2}{\gamma_{\min}^2}\nconv 0
\label{eq:connectivity}
\end{equation}

(note that subsequently, we will frequently omit the dependency of $N$ unless necessary for clarity).


This implies (from the second factor), that the $p_m$ are not of too different order which is a technicality needed throughout the proof.
The first factor guarantees either intra-block connectivity (i.e. all $p_m$ are sufficiently large) or inter-block connectivity (i.e. $q$ is sufficiently large). In the former case, connectivity of the entire graph is not guaranteed, so we need an additional assumption on $q$. Moreover, we require that $q$ must not be  too small compared to the $p_m$, resulting in the condition
\begin{align}
	((M-1)q)^2\gg \left(\frac{M\log N}{N}\right)^{1/4}p_{\min}^{5/4}p_{\max}^{1/2},
	\label{eq:qconditionneu}
\end{align}
where $p_{\max}=\max\limits_{m=1,\dots,M}p_m$ and $p_{\min}=\min\limits_{m=1,\dots,M}p_m$. 
The second factor in \eqref{eq:connectivity} and the dependence of the lower bound in \eqref{eq:qconditionneu} on the $p_m$
are artifacts of our subsequent proofs. 

As a matter of fact, together the two conditions imply that $G_N(M,P)$ is connected with probability converging to 1, as $N \to \infty$ (which we will denote by ''asymptotically 
almost-surely (a.a.s.)). Each block $i$ (without the edges leaving the
block) constitutes an Erd\H{o}s-R\'enyi graph on $n=N/M$ vertices and connection probability $p_i$. Then, $\frac{M \log^4(N)}{N\cdot p_m(N)\cdot (1-p_m(N))}\to \infty$ which is implied by \eqref{eq:connectivity} 
is (more than)
sufficient to ensure that block number $m$ is connected a.a.s. 

Moreover, 
note that there are $N/M$ vertices per block. \eqref{eq:connectivity} gives that all $p_m$ satisfy
$p_m \gg \frac{M \log^4 N}{N}$. Hence \eqref{eq:qconditionneu} implies that $q \gg \frac{M \log^{5/2} N}{(M-1) N}$.
This, together with the block sizes guarantees that all blocks are connected with each other, a.a.s., hence 
the entire graph is connected a.a.s. In the following results we will always tacitly assume that our realization of the stochastic block model is connected.

For fixed $N$ and a fixed realization $G=G_N$ of the stochastic block model, consider the simple random walk in discrete 
time $(X_t)$ on $G$: If $X_t$ is in $v\in V$ at time $t$, $X_{t+1}$ will be in $w$ with probability $\frac{1}{d_v}$, 
where $d_v$ denotes the degree of $v$, if $\{v,w\}\in E$, and with probability 0, otherwise. 

The invariant distribution of this walk is given by
\[\pi_v\coloneqq \frac{d_v}{\sum\limits_{w\in V}d_w}=\frac{d_v}{2|E|-|L|}.\]
Let $H_{vw}$ be the expected time it takes the walk to reach vertex $w\in V$ when starting from vertex $v\in V$. 
Of course, $H_{vw}$ will typically be sensitive to the choice of $v$ and $w$. To compensate for this, one averages over either $v$ or $w$. 
More precisely, we define
\begin{equation}\label{eq:defhitting}
	H_w:=\sum\limits_{v\in V}\pi_vH_{vw}\quad\text{ and }\quad H^v:=\sum\limits_{w\in V}\pi_w H_{vw}.
\end{equation}

Notice that in \cite{LT2020} and \cite{LT2022}, these are called \emph{average target hitting time} and \emph{average starting hitting time}, meanwhile, in \cite{LPW09}, the latter is called \emph{random target time}.
 Note that $H_w$ and $H^v$ are expectation values in the random walk measure, but with respect to the realization of the random graph, they are random variables.
In \cite{LoeweTorres}, the asymptotic behaviour of $H_w$ and $H^v$ were analyzed on the level of a law of large numbers for a random walk on a realization of an Erd\H{o}s-R\'enyi random graph. It was shown that
\[H_w=N(1+o(1))\quad\text{ as well as }\quad H^v=N(1+o(1)) \]
asymptotically almost surely proving a conjecture from \cite{sood}. The results were extended to central limit theorems in \cite{LT2020} and \cite{LT2022}. In \cite{HL19} the law of large numbers was extended to random hypergraphs. 
The aim of this paper is to prove the correspondings laws of large numbers  and a central limit theorem for $H_w$ for random graphs according to the stochastic block model.

To state these, define
\begin{equation}\label{eq:gamma_m}
\gamma_m:=\frac{N}{M}p_m+(M-1)\frac{N}{M}q \quad \mbox{for $m=1,\dots,M$.}\end{equation}
Note that 
$\gamma_m$ is the expected degree of a vertex in block $m$. 
Moreover, let $$\gamma_{\min}\coloneqq \min\limits_{m=1,\dots,M}\gamma_m.$$
Our goal is to prove the following statement:
\begin{thm}\label{maintheo}
	Assume that conditions \eqref{eq:connectivity} and \eqref{eq:qconditionneu} hold. Then
	\[H^v = N(1+o(1))\]
	asymptotically almost surely.
\end{thm}
\begin{thm}\label{maintheo2}
	Assume that conditions \eqref{eq:connectivity} and \eqref{eq:qconditionneu} hold. 
	Then, 
	\begin{align*}
		H_w
		=\frac{N}{M}\frac{\sum\limits_{m=1}^M\gamma_m}{\gamma_{B(w)}}(1+o(1))
	\end{align*}
	asymptotically almost surely. Recall that, for $w\in V$, the quantity ${B(w)}$ denotes the block, in which $w$ can be found.
\end{thm}

Under some additional assumptions, we can extend our result for $H_w$ to a central limit theorem. To that end, we introduce notation for average probability for an intra-block edge and average degree of a vertex:
\begin{align*}
	\bar p &:=\frac1M\sum\limits_{m=1}^Mp_m\quad\text{and}\quad
	\bar\gamma :=\frac1M\sum\limits_{m=1}^M\gamma_m
\end{align*}
Finally, additionally to $\gamma_m$, we denote the variance of the degree of a vertex in block $m$ and the average of these variances by
\begin{align*}
	\upsilon_m^2&\coloneqq\frac{N}{M}p_m(1-p_m)+(M-1)\frac NMq(1-q)\quad\text{and}\quad
	\bar\upsilon^2:=\frac{1}{M}\sum\limits_{m=1}^M\upsilon_m^2.
\end{align*}	
	In order to obtain a central limit theorem, we must assume that our graph is not too densely connected (otherwise we obtain a near-complete graph which leads to a near-deterministic hitting time with little fluctuations): To that end, we replace the condition
	\eqref{eq:connectivity}
	with
	\begin{equation}
		\frac{M \log^4(N)}{N p_m(N) (1-p_m(N))+N(M-1)q(N)(1-q(N))}\cdot\frac{p_{\max}^2}{p_{\min}^2}\nconv 0
		\label{eq:connectivityclt}
	\end{equation}

	Furthermore, we replace condition \eqref{eq:qconditionneu} by the following conditions:
	\begin{equation}
		\frac{\gamma_m}{\bar \gamma}\ll\frac{\upsilon_m}{\bar p},
		\label{eq:cltadditionalassump}
	\end{equation}
	as well as 
	\begin{equation}
		\frac{\bar\upsilon}{\upsilon_m}\cdot\frac{\gamma_m}{\bar \gamma}\ll\sqrt N,
		\label{eq:cltadditionalassump2}
	\end{equation}
	and		
	\begin{equation}
		\frac{\gamma_m}{\upsilon_m}\cdot\frac{p_{\min}^2}{\gamma_{\min}(M-1)^2q^2}\nconv 0
		\label{eq:spectralcondition_diffp}
	\end{equation}
	The conditions \eqref{eq:spectralcondition_diffp} is a stronger, more general version of \eqref{eq:qconditionneu} suited to a central limit theorem (i.e. also ensuring that in the case of very dense blocks, $q$ is not too large, which would lead to a near complete graph again). The other two (additional)
	conditions in the statement ensure that the $p_m$ are ``well balanced'' in the sense that the expectation of the degree of $w$ and the average degree in the graph are not too far apart. In the case of $(M-1)q\ll p_{\min}$, this is (up to a logarithmic factor) implied by \eqref{eq:connectivityclt}.
	
	Similarly, we want to ensure that the variances of the arising degrees are not too different to control the fluctuations, which is why we require that \eqref{eq:cltadditionalassump2} holds. Again, this conditions is up to logarithmic terms already implied by  the previous conditions in the strongly assortative setting. 
	
	Notice that both conditions are satisfied e.g. when all $p_m$ are constant, all of the same order, or all dominated by $(M-1)q$. 

Then, we can show the following theorem
\begin{thm}
	\label{thm:clttargethittingtimediffp}
	Assume that conditions \eqref{eq:connectivityclt} -- \eqref{eq:spectralcondition_diffp} hold. Let $w\in V$ be a vertex with block number $m=B(w)$ such that additionally
	hold.
	Then
	\begin{align*}
		\frac{\gamma_m^2}{N\upsilon_m\bar\gamma}\cdot\left(H_w-\frac{N\bar\gamma}{\gamma_m}\right)&\dconv \mathcal{N}(0,1),
	\end{align*} 	
	where $\dconv $ denotes convergence in distribution.
\end{thm}

\begin{rem}
	The results remains true if the number of blocks $M\coloneqq M(N)\ll N$ depends on the number of vertices $N$, e.g. $M(N)=\log N$, as long as all other conditions remain intact.
\end{rem}

%

The rest of this note is organized in the following way: In spectral graph theory, a common approach to average  hitting times (defined as above) is via 
a representation in terms of the eigenvalues and eigenvectors of a version of the Laplacian matrix of the graph 
(see e.g\ \cite{Lov96}). 
This matrix is 
closely related to the adjaceny matrix of the graph. In a first step we therefore bound the eigenvalues of the expected adjacency 
matrix in the stochastic block model. This will be done in Section 2 using estimates for the matrix of the block transition
probabilities. In Section 3 the considerations of Section 2 will result
in asymptotic results for the eigenvalues of the adjacency matrix in the stochastic block model. As one can already learn from 
Erd\H{o}s-R\'enyi graphs spectral gap estimates for the adjacency matrix and the Laplacian matrix 
play an important role for the order of magnitude of $H^i$.
These will be given in Section 4. Finally, Section 5 contains the proof of Theorem \ref{maintheo}, while in Section 6 we will prove
Theorem \ref{maintheo2}. 
In Section \ref{sec:identicalpi} we will then turn to the case of diagonal entries $p_1=\dots=p_M=p$ of the diagonal of $P_M$. In that case, we are able to improve the conditions for our results substantically through relatively simple modifications of the proofs and additionally prove a central limit theorem.

We separate the problem into several distinct cases: Let
\begin{align}\label{eq:kappa}
	\kappa\coloneqq\lim\limits_{N\to\infty} \frac{(M-1)q}{p_{\min}}\in[0,\infty]
\end{align}
Throughout, we will always require that $\kappa$ exists.
The expression $\kappa$ can be thought of as a quantification of assortativity in the model: For $\kappa=0$, the model is very clearly (strongly) assortative, i.e.\ the intra-block connection probabilities are larger than the connection probabilities between
blocks.
On the other hand, $\kappa=\infty$ does not guarantee a disassortative model
(i.e.\ all the intra-block connection probabilities are smaller than the connection probabilities between
blocks), yet guarantees that the edges within blocks are irrelevant on a macroscopic scale: The connectivity between different blocks dominates.

\section{Eigenvalues of the expected adjusted adjacency matrix}

In this section we will discuss the spectrum of the expected adjacency matrix of a realization of the stochastic block model. 

We begin by rescaling the transition probabilities suitably. Denote by $\Gamma$ the diagonal matrix containing the expected degrees of vertices in each block, i.e. $\Gamma=\mathrm{diag}(\gamma_m)$. 
Then we rescale the matrix of transition probabilities as follows:
\[P_M'=\Gamma^{-1/2}P_M\Gamma^{-1/2},\qquad\text{i.e. }p_{i,j}'=\frac{p_{i,j}}{\sqrt{\gamma_m\gamma_l}}.\]

Let us denote the (random) adjacency matrix of our random graph $G_N$ by $A_N$, i.e.\ $A_N=(a_{v,w})_{v,w \in V_N}$ and 
$a_{v,w}=1$, if $\{v,w\} \in E_N$, otherwise $a_{v,w}=0$.

The corresponding rescaled version of $A$ is given by
\begin{equation}A'_N=(\E D)^{-1/2}A_N(\E D)^{-1/2},\qquad\text{i.e. }a_{v,w}'=\frac{a_{v,w}}{\sqrt{\E d_v \E d_w}}, \label{eq:AN'} \end{equation}
where $\E D$ denotes the $N$-dimensional diagonal matrix consisting of the expected degrees of the $N$ vertices.

Not too surprisingly, it will turn out that the spectrum of $\E A_N'$ is closely related to the spectrum of $P_M'$. We will 
therefore start with a quick analysis of the eigenvalues of $P_M'$.


Through graph conductance we can obtain an upper bound on the spectral gap of $P_M'$:
\begin{lem}\label{lem:disassortativeP}
\[\frac NM \lambda_2(P_M')\leq 1-{2}\cdot{\left(1+\frac{p_M}{(M-1)q}\right)^{-2}},\] i.e. the rescaled second eigenvalue of $P_M'$ is bounded away from 1.
\end{lem}
\begin{proof}
	We rewrite 
	\[P_M'=\Gamma^{-1/2}P_M\Gamma^{-1/2}=\frac MN\Theta^{-1/2}P_M\Theta^{-1/2}\]
	with $$\Theta=\mathrm{diag}(p_m+(M-1)q)=\frac MN\mathrm{diag}(\gamma_m).$$
	Then estimating the eigenvalues of $\Theta^{-1/2}P_M\Theta^{-1/2}$ appropriately is sufficient to obtain the claim.
	
	We consider $P_M$ as the weight matrix of the weighted, undirected graph $G_P$ on the vertex set $\{1,\dots,M\}$, i.e.\ the graph is complete where each edge between different vertices is assigned weight $q$ and each vertex $m\in\{1,\dots,M\}$ has a loop with weight $p_m$, i.e. the weight function $w$ is defined by $w(m,l)=p_m\mathbbm{1}_{m=l}+q\mathbbm{1}_{m\neq l}.$ 
	The degrees of the vertices in this graph are given by the sum of weights of adjacent edges, i.e. 
$$\mathrm{deg}(m)=p_m+(M-1)q=\Theta_{m,m}.$$ Hence,
 $\Theta^{-1/2}P_M\Theta^{-1/2}$is the symmetrically normalized adjacency matrix of $G_P$. Let us further denote by $\mathcal{L}_P\coloneqq I-\Theta^{-1/2}P_M\Theta^{-1/2}$ the symmetrically normalized Laplacian of $G_P$. We are going to compute the spectral gap of $\mathcal L_P$ using Cheeger's inequality on the weighted conductance of $G_P$
(cf. \cite[Section 3]{DS91}, e.g.). 
	To that end, for a weighted graph $G=(V,E,w)$ let us define for a subset $S\subseteq V$ with edge boundary $\partial S$ (i.e. the set of edges $\{m,l\}\in E$ such that $m\in S$ and $l\notin S$)
	\begin{align*}
		\mathrm{vol}_w(S)&\coloneqq\sum\limits_{m\in S}\deg_w(m)\\
		\phi_w(S)&\coloneqq \frac{1}{\mathrm{vol}_w(S)}\sum\limits_{e\in \partial S}w(e)\\
		\phi_w(G)&\coloneqq \min\limits_{S:\mathrm{vol}_w(S)\leq \frac12 \mathrm{vol}_w(V)}\phi_w(S)
	\end{align*}
This latter quantity this called the conductance of $G$.	
In our situation, since we assumed that $p_1\geq \dots \geq p_M$, we obtain that 
	\begin{align*}
		\phi_w(G_P)&\leq \phi_w(\{M\})=\frac{1}{\mathrm{\deg}(M)}\sum\limits_{m\neq M}w(m,M)=\frac{1}{p_M+(M-1)q}\cdot (M-1)q=\frac{1}{1+\frac{p_M}{(M-1)q}}
	\end{align*}
	Cheeger's inequality states that the second smallest eigenvalue of $\mathcal L_P$ is bounded from below by twice the square of the conductance of the graph, i.e. 
	\begin{align}
		\lambda_{M-1}(\mathcal L_P)&\geq 2\phi_w^2(G_P)=2\cdot{\left(1+\frac{p_M}{(M-1)q}\right)^{-2}}
	\end{align}
	Consequentially,
	\begin{equation*}
		\frac NM \lambda_2(P_M')=1-\lambda_{M-1}(\mathcal L_P)\leq 1-{2}\cdot {\left(1+\frac{p_M}{(M-1)q}\right)^{-2}}. \qedhere
	\end{equation*}
\end{proof}

Next consider the expected adjacency matrix $\E A_N'$ (recall \eqref{eq:AN'} for its definition) where we take expectation entry-wise.   
Then, $\E A_N'$ consists of entries from $P_M'$ depending on the blocks the vertices are in. 

More precisely, 
$\E a_{v,w}'=\frac{p_m}{\gamma_m}$, if both $v$ and $w$
are in block $m, m=1, \ldots M$, otherwise $\E a_{v,w}'=\frac{q}{\sqrt{\gamma_{B(v)}\gamma_{B(w)}}}$. 
Not surprisingly, the spectrum of $\E A_N'$ can be simply derived from the eigenvalues of $P_M$:
\begin{lem}\label{lem:evexpmatrix}
The matrix $\E A_N'$ has exactly $M$ non-zero eigenvalues. These are given by 
\begin{equation}\label{eq:EVEA}
\frac{N}{M}\lambda_m(P_M'), \qquad m=1,\dots,M.
\end{equation}
\end{lem}
\begin{proof}
Denoting by $J$ the $(N/M)\times (N/M)$ matrix in which each entry equals $1$, we see that 
$\E A_N'=P_M'\otimes J$, where $\otimes$ denotes the Kronecker product of two matrices matrices. 

Since $J$ has only one non-zero eigenvalue, which is $\frac{N}{M}$, 
we obtain that the $M$ non-zero eigenvalues of the matrix $\E A_N'$ are given by \eqref{eq:EVEA} as claimed.
\end{proof}

\section{Eigenvalues of the rescaled adjacency matrix}
Let us now turn to the rescaled adjacency matrix $A_N'$ (again recall \eqref{eq:AN'} for its definition). 
More precisely, we consider the centered rescaled adjacency matrix of $G_N$, 
$$X\coloneqq A_N'-\E A_N', $$ first.

Denote by
\begin{equation}\label{eq:defsigma}
\sigma^2\coloneqq\frac{1}{M}(\max\limits_{m=1,\dots,M}p_m(1-p_m)+(M-1)q(1-q)).
\end{equation}
Then we can give an upper bound on the maximum eigenvalue (in absolute value) of the matrix $X$:
\begin{thm}\label{thm:Xmaxeigenvalue}
There is a constant $c$ such that
\[\|X\|_2\leq \frac{1}{\gamma_{\min}}\Big(2\sqrt{N\sigma^2}+c\log(N)\sqrt[4]{N\sigma^2}\Big)\]
holds asymptotically almost surely.
\end{thm}

\begin{proof}

We modify the proof of \cite{avrachenkovinria}, using the original approach by \cite{Vu2007} to our model.

The core idea is to apply Wigner's trace method and estimate the terms in the trace of $X^k$ appropriately. 
To this end, we denote by $X_{v,w}, v,w= 1, \ldots ,N$ the entries of the matrix $X$. Note that each of the $X_{v,w}$ can only take two values:
Either $X_{v,w}=1-p_m$ and this happens with probability $p_m$ or $X_{v,w}=-p_m$ and this event 
has probability $1-p_m$. Here $m$ is the block number, if both $v$ and $w$ are in the $m$-th block, and $m=0$, otherwise. 
We denote by 
\[\sigma_m^2=p_m(1-p_m)\]
the variance of these entries, depending on the edge's block number $m\in\{0,\dots,M\}$ 

Now clearly for an even integer $k$, we have
\begin{equation}
\EW{\mathrm{tr}\left(X^k\right)}=\sum\limits_{v_1,\dots,v_k\in V}\EW{X_{v_1,v_2}\cdots X_{v_{k-1},v_k}\cdot X_{v_k,v_1}}
\label{eq:traceX}
\end{equation}

We interpret the sequence $I\coloneqq(v_1,\dots,v_k,v_1)$ as a circular path on the complete graph $K_N$ with edge weights $X_{v,w}$ for 
an (undirected) edge $(v,w)\in V_N^2$. 

Now each ordered sequence of vertices that constitutes $I=(v_1,\dots,v_k,v_1)$ corresponds to a number of {\it {distinct}} vertices 
appearing in this sequence, $$s_I\coloneqq|S_I|\coloneqq|\{v_1,\dots,v_k\}|$$ and a permutation of $S_I$, giving the order in which the 
vertices appear for the first time in the sequence: $J_I\coloneqq(w_1,\dots,w_{s_I})$ such that  
$w_1,\dots,w_{s_I}\in S_I$, $w_l\neq w_m$ for 
all $l\neq m$. 

Notice that the expectation of the weights of a path is 0 if there is at least one edge $e$ with multiplicity 1 (i.e.\ $e$ only appears 
once within $\{v_1,v_2\},\dots\{v_{k-1},v_k\},\{v_k,v_1\}$). We thus only consider paths in which each edge appears with multiplicity 
at least two. 
Because the path has $k$ edges in total, the number of distinct edges, $\nu_I$, 
in the path is bounded by $k/2$, as $k$ is even. 
In a circular path, the number of distinct vertices is at most the number of distinct edges. Hence we also have $s_I\leq k/2$.

Now, for $2\leq x\leq s_I$ denote by $e_x^{I}$ the first edge $(v_l,v_{l+1})$ on the path $I$ such that $v_{l+1}=w_x$ 
(i.e. the first edge that discovers the vertex $w_x$, the vertex $w_1$ is excluded as the starting vertex of the path).

Notice that the absolute value of the entries of $X$ is bounded by $$K\coloneqq\frac{1}{\gamma_{\min}}\max\limits_{m=0,\dots,M}p_m(1-p_m).$$ 

We use this to remove edges from the product of weights of the path, keeping only ``discovery'' edges $e_2^{I},\dots,e_{s_I}^{I}$ with multiplicity 2 each:
\[\EW{X_{v_1,v_2}\cdots X_{v_{k-1},v_k}\cdot X_{v_k,v_1}}\leq\frac{1}{\gamma_{\min}^{k-2(s_I-1)}} K^{k-2(s_I-1)} \EW{\prod\limits_{x=2}^{s_I}X_{e_x^{I}}^2}.\]

We denote for an edge $e=\{v,w\}$ by
\[B(e)=\begin{cases}
	m,&\text{if }v,w\in V^{(m)}\\
	0,&\text{if }v\in V^{(m)},w\in V^{(l)},\, l\neq m
\end{cases}
\]
the block number, if $v$ and $w$ are in the same block, and 0 otherwise.

As the discovery edges are clearly pairwise different (as each discovers a new vertex), we find
\[\EW{\prod\limits_{x=2}^{s_I}X_{e_x^{I}}^2}=\prod\limits_{x=2}^{s_I}\EW{X_{e_x^{I}}^2}=\frac{1}{\gamma_{\min}^{2(s_I-1)}}\prod\limits_{x=2}^{s_I}\sigma_{B(e_x^I)}^2.\]

Hence, we can write from \cref{eq:traceX}
\begin{equation*}
\EW{\mathrm{tr}\left(X^k\right)}\leq\frac{1}{\gamma_{\min}^{k}}\sum\limits_{s=1}^{k/2}K^{k-2(s_I-1)}\sum\limits_{\substack{I=(v_1,\dots,v_k)\\s_I=s}}\prod\limits_{x=2}^{s_I}\sigma_{B(e_x^I)}^2.
\end{equation*}


Let $s$ be fixed.

To bound the inner sum, we choose the discovery vertices of the path successively. For the first one, we have $N$ options. 
For each following one $2\leq x\leq s$ we have two scenarios: 
\begin{itemize}
	\item If the associated $x$-th discovery edge $e_x$ satisfies $B(e_x)=m$ for some $m=1,\dots,M$ (i.e. the vertex discovery happens from the same block), then $\sigma_{B(e_x)}^2=\sigma_{m}^2\leq \max\limits_{l=1,\dots,M}\sigma_l^2\eqqcolon \sigma_*^2$, and there are $N/M$ possibilities to choose the $x$-th discovery vertex.
	\item If the associated $x$-th discovery edge $e_x$ satisfies $B(e_x)=0$ (i.e. the vertex discovery happens from a different block), then $\sigma_{B(e_x)}^2=\sigma_{0}^2$ and there are $(M-1)\cdot N/M$ possibilities to choose the $x$-th discovery vertex.
\end{itemize}
Now that all options for the discovery sequence are exhausted, it remains to bound the number of paths with this exact discovery 
sequence which also satisfies the same-block and different block-discoveries, respectively, as chosen above. This, however, is bounded 
from above by the number of paths on $K_s$ with this exact discovery sequence (i.e. we ignore the block restrictions for discovery). 
This, in turn, is independent of the particular discovery sequence, but only depends on the length of the path $k$ and the length of the 
discovery sequence $s$.  An upper bound for this is given by
\begin{equation}\label{eq:defW'}
W'(k,s)\coloneqq \binom{k}{2s-2}2^{2k-2s+3}s^{k-2s+2}(k-2s+4)^{k-2s+2}
\end{equation}
according to \cite[Lemma 4.1]{Vu2007}.

Overall, we thus obtain by differentiating over the number of same-block discoveries $t$
\begin{align}
\EW{\mathrm{tr}\left(X^k\right)}&\leq\frac{1}{\gamma_{\min}^{k}}\sum\limits_{s=1}^{k/2}K^{k-2(s-1)}\cdot N\cdot\sum\limits_{t=0}^{s-1} \binom{s-1}{t}\left(\frac{N}{M}\sigma_*^2\right)^t\left((M-1)\frac{N}{M}\sigma_0^2\right)^{s-1-t}\cdot W'(k,s)\notag\\
&\leq\frac{1}{\gamma_{\min}^{k}}\sum\limits_{s=1}^{k/2}K^{k-2(s-1)}\cdot N\cdot\left(\frac{N}{M}\right)^{s-1}\left(\sigma_*^2+(M-1)\sigma_0^2\right)^{s-1}\cdot W'(k,s)\notag\\
&=\frac{1}{\gamma_{\min}^{k}}\sum\limits_{s=1}^{k/2} S(N,k, s).
\label{eq:traceX2}
\end{align}
with
$$
S(N,k, s):=K^{k-2(s-1)}\cdot N\cdot\left(\frac{N}{M}\right)^{s-1}\left(\sigma_*^2+(M-1)\sigma_0^2\right)^{s-1}\cdot W'(k,s).
$$
Now notice that by the definition of $\sigma^2$ \eqref{eq:defsigma} and by the definition of the $W'$-terms
\eqref{eq:defW'} we have for $s=1,\dots,k/2-1$
\begin{align*}
\frac{S(N,k,s)}{S(N,k,s+1)}\hspace{-2cm}&\hspace{2cm}=\frac{K^2}{\sigma^2}\cdot\frac{1}{N}\cdot\frac{W'(k,s)}{W'(k,s+1)}\\
&=\frac{K^2}{N\sigma^2}\cdot\frac{\binom{k}{2s-2}}{\binom{k}{2s}}\cdot\frac{2^{2k-2s+3}}{2^{2k-2s+1}}\cdot\frac{s^{k-2s+2}}{(s+1)^{k-2s}}\cdot\frac{(k-2s+4)^{k-2s+2}}{(k-2s+2)^{k-2s}}\\
&=\frac{K^2}{N\sigma^2}\cdot \frac{2s(2s-1)}{(k+2-2s)(k+1-2s)}\cdot 4\cdot \frac{s^{k-2s+2}}{(s+1)^{k-2s}}\cdot(k-2s+4)^{2}\cdot\left(1+\frac{2}{k-2s+2}\right)^{k-2s}\\
&\leq\frac{K^2}{N\sigma^2}\cdot \frac{2s(2s-1)}{(k+2-2s)(k+1-2s)}\cdot 4\cdot s^2\cdot(k-2s+4)^{2}\cdot\left(\left(1+\frac{1}{k/2-s}\right)^{k/2-s}\right)^2\\
&\leq\frac{8K^2s^3(2s-1)}{N\sigma^2}\cdot \frac{(k-2s+4)^{2}}{(k+2-2s)(k+1-2s)}\cdot e^2\\
&\leq\frac{8e^2K^2s^3(2s-1)}{N\sigma^2}\cdot \left(1+\frac{2}{k+2-2s}\right)\left(1+\frac{3}{k+1-2s}\right)\\
&\leq\frac{8e^2K^2s^3(2s-1)}{N\sigma^2}\cdot 2\cdot \frac{5}{2}\\
&\leq\frac{80e^2K^2s^4}{N\sigma^2}\\
&\leq\frac{5e^2K^2k^4}{N\sigma^2}
\end{align*}
 Hence, one can easily see that for 
\begin{equation}
k\leq \sqrt[4]{\frac{N}{10e^2}\cdot\frac{\sigma^2}{K^2}}
\label{eq:kbound}
\end{equation} we obtain
\[\frac{S(N,k,s)}{S(N,k,s+1)}\leq\frac{1}{2}.\] 
Therefore, by geometric series and \cref{eq:traceX2}
\begin{align*}
\EW{\mathrm{tr}(X^k)}&\leq \frac{1}{\gamma_{\min}^{k}} S(N,k,k/2+1)=\frac{1}{\gamma_{\min}^{k}}N\cdot N^{k/2}(\sigma^2)^{k/2}\cdot 2^{k+1}=2N\left(2\frac{\sigma}{\gamma}\sqrt{N}\right)^k
\end{align*}
Let $k=\sqrt[4]{\frac{N\sigma^2}{10e^2K^2}}$ is an even integer, otherwise take the largest even integer bounded by \cref{eq:kbound}). Due to 
$$\|X\|_2^k=\max\limits_{m=1,\dots,N} |\lambda_m(X)|^k\leq\sum\limits_{m=1}^N\lambda_m(X)^k=\mathrm{tr}(X^k)$$ 
we obtain for some constant $c\in\R$ by a  high moment Markov inequality 
\begin{align*}
&\WK{\|X\|_2\geq \frac{1}{\gamma_{\min}}\Big(2\sigma\sqrt{N}+c\sigma^{1/2}\log(N)\cdot N^{1/4}\Big)}\\\leq& \frac{\gamma_{\min}^k}{(2\sigma\sqrt{N}+c\sigma^{1/2}\log(N)\cdot N^{1/4})^k}\EW{\mathrm{tr}(X^k)}\\
\leq & \frac{\gamma_{\min}^k \cdot 2N\left(2\frac{\sigma}{\gamma_{\min}}\sqrt{N}\right)^k}{(2\sigma\sqrt{N}+c\sigma^{1/2}\log(N)\cdot N^{1/4})^k}\\
\leq & 2N\left(1-\frac{c\sigma^{1/2}\log(N)\cdot N^{1/4}}{2\sigma\sqrt{N}+c\sigma^{1/2}\log(N)\cdot N^{1/4}}\right)^k\\
\leq & 2N\exp\left(-\frac{c\sigma^{1/2}\log(N)\cdot N^{1/4}\cdot k}{2\sigma\sqrt{N}+c\sigma^{1/2}\log(N)\cdot N^{1/4}}\right)\\
= & 2N\exp\left(-\frac{c\log(N)\cdot N^{1/2}\cdot \sqrt[4]{\frac{\sigma^2}{10e^2K^2}}}{2\sigma\sqrt{N}+c\sigma^{1/2}\log(N)\cdot N^{1/4}}\right)\\
\leq & 2N\exp\left(-c\sigma^{1/2}\frac{\log(N)\sqrt[4]{\frac{\sigma^2}{10e^2K^2}}}{2\sigma+c\sigma^{1/2}\frac{\log(N)}{N^{1/4}}}\right)\\
\leq & 2N\exp\left(-c\sigma^{1/2}\frac{\log(N)\sqrt[4]{\frac{\sigma^2}{10e^2K^2}}}{3\sigma}\right)\\
= & 2N\exp\left(-c\frac{\log(N)}{3\sqrt[4]{10e^2}\sqrt{K}}\right)\\
= & 2N\exp\left(-c\frac{2\log(N)}{3\sqrt[4]{10e^2}}\right)\\
\end{align*}
where we use \cref{eq:connectivity} and $K\leq\frac14$. For sufficiently large constant $c$, the latter term converges to 0.
\end{proof}

\begin{rem}\label{rem:asymptfirstorder}
Together with \cref{eq:connectivity}, the statement of \cref{thm:Xmaxeigenvalue} can be used to obtain
\[\|X\|_2\leq \frac{2}{\gamma_{\min}}\sqrt{N\sigma^2}(1+o(1))\]
asymptotically almost surely. Indeed, note that this is the case if and only if 
$$
\sqrt[4]{N\sigma^2} \gg \log N
$$ which follows 
from \eqref{eq:connectivity} immediately if $p$ is bounded away from 1.

Moreover, notice that this also gives uniform bound on the eigenvalues of $X$ since the spectral norm is sub-multiplicative. 
\end{rem}

We can finally give the asymptotically leading order of the eigenvalues of the symmetrically rescaled adjacency matrix $A'$:

\begin{prop}\label{prop:evA}
The $M$ largest eigenvalues of $A'$ are given by
\[\frac{N}{M}\lambda_k(P_M')+O\left(\frac{\sqrt{N\sigma^2}+\log(N)\sqrt[4]{N\sigma^2}}{\gamma_{\min}}\right),\] $k=1,\dots,M$, with probability converging to 1. All other eigenvalues of $A$ are bounded by
\[\lambda_k(A')=O\left(\frac{\sqrt{N\sigma^2}+\log(N)\sqrt[4]{N\sigma^2}}{\gamma_{\min}}\right)\]
for $k>M$, with probability converging to 1.
\end{prop}

\begin{proof}
For a $N\times N$-matrix $R$, denote by $\lambda_1(R)\geq\lambda_2(R)\geq\dots\geq\lambda_N(R)$ its eigenvalues in descending order.
As $A'=X+\E A'$, by Weyl's interlacement theorem \cite[Theorem 10.3.1]{Parlett}
we have for $i,j=1, \dots,N$ such that $i+j\leq N+1$
\[\lambda_{N+1-i}(X)+\lambda_{N+1-j}(\E A')\leq \lambda_{N+2-i-j}(A')\text{ and }\lambda_{i+j-1}(A')\leq\lambda_{i}(X)+\lambda_{j}(\E A')\]

From Theorem \ref{thm:Xmaxeigenvalue}, we deduce that $\lambda_i(X)= O\left(\frac{\sqrt{N\sigma^2}+\log(N)\sqrt[4]{N\sigma^2}}{\gamma_{\min}}\right)$  asymptotically almost surely for $i=1, \dots,N$. Therefore,
\[O\left(\frac{\sqrt{N\sigma^2}+\log(N)\sqrt[4]{N\sigma^2}}{\gamma_{\min}}\right)+\lambda_{N+1-j}(\E A')\leq \lambda_{N+2-i-j}(A')\] and \[\lambda_{i+j-1}(A')\leq O\left(\frac{\sqrt{N\sigma^2}+\log(N)\sqrt[4]{N\sigma^2}}{\gamma_{\min}}\right)+\lambda_{j}(\E A')\]
asymptotically almost surely.
With $i=1$, $j=N+1-k$ in the first inequality and $j=k$, $i=1$ in the second, we obtain for $k=1,\dots,N$
\[O\left(\frac{\sqrt{N\sigma^2}+\log(N)\sqrt[4]{N\sigma^2}}{\gamma_{\min}}\right)+\lambda_{k}(\E A')\leq \lambda_{k}(A')\] and \[\lambda_{k}(A')\leq O\left(\frac{\sqrt{N\sigma^2}+\log(N)\sqrt[4]{N\sigma^2}}{\gamma_{\min}}\right)+\lambda_{k}(\E A')\]
and consequently,
\begin{equation}
\lambda_k(A')=\lambda_k(\E A')+O\left(\frac{\sqrt{N\sigma^2}+\log(N)\sqrt[4]{N\sigma^2}}{\gamma_{\min}}\right),
\label{eq:eigenvaluesAandmean}
\end{equation}
with probability converging to 1.
Together with Lemma \ref{lem:evexpmatrix} we obtain the claim.
\end{proof}


\section{Spectral gap of the symmetric normalized adjacency matrix}

We now consider the matrix $B=D^{-1/2}AD^{-1/2}$, where $D=\mathrm{diag}(d_1,\dots,d_n)$ and $d_i$ denotes the degree of the vertex $i$.
Since both the average hitting times have a decomposition in terms of the eigenvalues (and eigenvectors) of $B$ we 
are interested in bounding the spectrum of $B$.
Recall that 
$\gamma_m=\E d_v=\frac{N}{M}(p_m+(M-1)q)$
for $v\in V^{(m)}$, $m=1,\dots,M$. 
Furthermore, by Chernoff's inequality (c.f. Theorem 2.4 in \cite{ChungLu}), we find that
\begin{align}
	\begin{split}
	d_v>\gamma_m-c\sqrt{\gamma_m}&\quad\text{with probability }1-\exp\left(-\frac{c^2}{2}\right)\\
	d_v<\gamma_m+c\sqrt{\gamma_m}&\quad\text{with probability }1-\exp\left(-\frac{c^2}{2\left(1+\frac{c}{3\sqrt{\gamma_m}}\right)}\right)
	\end{split}
	\label{eq:degreechernoff}
\end{align}

Denote $$R:=B-A'.$$ It is well known that the spectral radius $\|R\|_2$ is bounded by all consistent matrix norms and thus in particular $\|R\|_2\leq\|R\|_\infty$.

\begin{lem}\label{lem:Reigenvalues}
For $R$ as above, we have
\[\|R\|_\infty\leq \sqrt[4]{\frac{\log N}{\gamma_{\min}}}\sqrt{\frac{p_{\max}}{p_{\min}}}\cdot(\sqrt{3}+o(1))\]
asymptotically almost surely.  
\end{lem}
\begin{proof} The proof is similar to that of Lemma 3.3 in \cite{LoeweTorres} but given here for the sake of completeness.
Clearly,
the entries of $R$ are given by $$r_{v,w}=\frac{\sqrt{\gamma_m\gamma_l}-\sqrt{d_md_l}}{\sqrt{\gamma_m\gamma_l d_vd_w}}a_{v,w}$$ for $v\in V^{(m)}$, $w\in V^{(l)}$.
By setting $c={\sqrt{\log N}}$ in \cref{eq:degreechernoff}, 
$$|d_v-\gamma_m|\leq  \sqrt{\log N\gamma_m} \quad \mbox{for all $m\in\{1,\dots,M\}$ and all $v\in V^{(m)}$} 
$$ 
with probability converging to 1. Hence,
\begin{align*}
	|d_vd_w-\gamma_m\gamma_l|&=|d_v(d_w-\gamma_m)+\gamma_l(d_v-\gamma_m)|\\
	&=|(d_v-\gamma_m)(d_w-\gamma_l)|+\gamma_m|d_w-\gamma_l|+\gamma_l|d_v-\gamma_m|\\
	&\leq\sqrt{\log N\cdot\gamma_m}\sqrt{\log N\cdot\gamma_l}+\gamma_m\sqrt{\log N\cdot\gamma_l}+\gamma_l\sqrt{\log N\cdot\gamma_m}\\
	&=\log N\sqrt{\gamma_m\gamma_l}+2\sqrt{\max(\gamma_m,\gamma_l)\log N}\sqrt{\gamma_m\gamma_l}\\
	&\leq 3\sqrt{\log N \gamma_{\max}\gamma_m\gamma_l}
\end{align*}
and thus
\[|\sqrt{\gamma_m\gamma_l}-\sqrt{d_vd_w}|\leq \sqrt{|d_vd_w-\gamma_m\gamma_l|}\leq  \sqrt3\sqrt[4]{\log N\gamma_{\max}\gamma_m\gamma_l}\]
with probability tending to 1.
Furthermore,
$$d_vd_w>(\gamma_m-\sqrt{\log N\gamma_m})(\gamma_l-\sqrt{\log N\gamma_l})$$ with probability converging to 1 for all $m,l\in\{1,\dots,M\}$ and $v\in V^{(m)}, w\in V^{(l)}$. Thus,
\begin{align*}
	\|R\|_\infty&=\max\limits_{\substack{m=1,\dots,M\\v\in V^{(m)}}}\sum\limits_{w\in V}|r_{v,w}|=\max\limits_{\substack{m=1,\dots,M\\v\in V^{(m)}}}\sum\limits_{\substack{l=1,\dots,M\\w\in V^{(l)}}}\left|\frac{\sqrt{\gamma_m\gamma_l}-
	\sqrt{d_vd_w}}{\sqrt{\gamma_m\gamma_ld_vd_w}}a_{v,w}\right|\\
	&\max\limits_{\substack{m=1,\dots,M\\v\in V^{(m)}}}\sqrt3\sqrt[4]{\log N\gamma_{\max}}\sum\limits_{\substack{l=1,\dots,M\\w\in V^{(l)}}}\frac{a_{v,w}}{\sqrt[4]{\gamma_m\gamma_l}\sqrt{d_vd_w}}\\
	&\leq\sqrt3\sqrt[4]{\log N\gamma_{\max}}   )\max\limits_{\substack{m=1,\dots,M\\v\in V^{(m)}}}\sum\limits_{\substack{l=1,\dots,M\\w\in V^{(l)}}}\frac{a_{v,w}}{\sqrt[4]{\gamma_m\gamma_l}\sqrt{(\gamma_m-\sqrt{\log N\gamma_m})(\gamma_l-\sqrt{\log N\gamma_l})}}\\
	&\leq\sqrt3\sqrt[4]{\log N\gamma_{\max}} \max\limits_{m,l=1,\dots,M}\frac{1}{\sqrt[4]{\gamma_l}\sqrt{\gamma_l-\sqrt{\log N\gamma_l}}}\frac{1}{\sqrt[4]{\gamma_m}\sqrt{\gamma_m-\sqrt{\log N\gamma_m}}}\max\limits_{v\in V^{(m)}}\sum\limits_{\substack{l=1,\dots,M\\w\in V^{(l)}}}a_{v,w}\\
	&\leq\sqrt3\sqrt[4]{\log N\gamma_{\max}}  \max\limits_{m,l=1,\dots,M}\frac{1}{\sqrt[4]{\gamma_m^3\gamma_l^3}}\frac{1}{\sqrt{1-\sqrt{\frac{\log N}{\gamma_l}}}}\frac{1}{\sqrt{1-\sqrt{\frac{\log N}{\gamma_m}}}}\gamma_m\Big(1+\sqrt{\frac{\log N}{\gamma_m}}\Big)\\
	&\leq\sqrt3\sqrt[4]{\log N\gamma_{\max}} \max\limits_{m,l=1,\dots,M}\sqrt[4]{\frac{\gamma_m}{\gamma_l^3}}(1+o(1))\\	
	&\leq\sqrt3\sqrt[4]{\log N\frac{\gamma_{\max}^2}{\gamma_{\min}^3}}(1+o(1))\\
	&=\sqrt3\sqrt[4]{\frac{\log N}{\gamma_{\min}}}\sqrt{\frac{\gamma_{\max}}{\gamma_{\min}}}(1+o(1))
\end{align*}
with probability converging to 1, since by condition \eqref{eq:connectivity} $\gamma_m \gg \log^4(N)$ for all $i\in\{1,\dots,M\}$. 
\end{proof}

Next, consider the eigenvalues of $B$:

\begin{prop}\label{prop:eigenvaluesB}
The $M$ largest eigenvalues of $B$ are given by $\lambda_1(B)=1$
and
\begin{equation} \label{eq:nonPer} \lambda_k(B)=\frac NM\lambda_k(P_M')+O\left(\sqrt[4]{\gamma_{\min}^{-1}\log N}\sqrt{\frac{\gamma_{\max}}{\gamma_{\min}}}\right)+O\left(\frac{\sqrt{N\sigma^2}+\log(N)\sqrt[4]{N\sigma^2}}{\gamma{\min}}\right),\end{equation}
$k=2,\dots,M$ asymptotically almost surely. 

The remaining eigenvalues are given by
\[\lambda_k(B)=O\left(\sqrt[4]{\gamma_{\min}^{-1}\log N}\sqrt{\frac{\gamma_{\max}}{\gamma_{\min}}}\right)+O\left(\frac{\sqrt{N\sigma^2}+\log(N)\sqrt[4]{N\sigma^2}}{\gamma_{\min}}\right)\]
for all $k=M+1,\dots,N$ asymptotically almost surely.
\end{prop}

\begin{proof}
Again, by Weyl's interlacement theorem \cite[Theorem 10.3.1]{Parlett} we obtain the inequalities
\[\lambda_{N+1-i}(R)+\lambda_{N+1-j}(A')\leq \lambda_{N+2-i-j}(B)\text{ and }\lambda_{i+j-1}(B)\leq\lambda_{i}(R)+\lambda_{j}(A')\]

Let $k\in\{1,\dots,N\}$.
Choosing $i=1$ and $j=N+1-k$  in the first inequality and $i=1$, $j=k$ in the second inequality, we obtain
\[\lambda_{N}(R)+\lambda_{k}(A')\leq \lambda_{k}(B)\text{ and }\lambda_{k}(B)\leq\lambda_{1}(R)+\lambda_{k}(A').\]
Thus, using \cref{prop:evA} and \cref{lem:Reigenvalues} we obtain for $m\in\{1,\dots,M\}$ 
\begin{align*}
\lambda_{k}(B)&\geq \frac{N}{M}\lambda_k(P_M')+O\left(\sqrt[4]{\gamma_{\min}^{-1}\log N}\sqrt{\frac{\gamma_{\max}}{\gamma_{\min}}}\right)+O\left(\frac{\sqrt{N\sigma^2}+\log(N)\sqrt[4]{N\sigma^2}}{\gamma_{\min}}\right)\\
\lambda_{k}(B)&\leq \frac{N}{M}\lambda_k(P_M')+O\left(\sqrt[4]{\gamma_{\min}^{-1}\log N}\sqrt{\frac{\gamma_{\max}}{\gamma_{\min}}}\right)+O\left(\frac{\sqrt{N\sigma^2}+\log(N)\sqrt[4]{N\sigma^2}}{\gamma_{\min}}\right)
\end{align*}
asymptotically almost surely. 
For $k\in\{M+1,\dots,N\}$ we arrive at
\begin{align*} 
\lambda_{k}(B)&\geq O\left(\sqrt[4]{\gamma_{\min}^{-1}\log N}\sqrt{\frac{\gamma_{\max}}{\gamma_{\min}}}\right)+O\left(\frac{\sqrt{N\sigma^2}+\log(N)\sqrt[4]{N\sigma^2}}{\gamma_{\min}}\right)\\
\lambda_{k}(B)&\leq O\left(\sqrt[4]{\gamma_{\min}^{-1}\log N}\sqrt{\frac{\gamma_{\max}}{\gamma_{\min}}}\right)+O\left(\frac{\sqrt{N\sigma^2}+\log(N)\sqrt[4]{N\sigma^2} }{\gamma_{\min}}\right).
\end{align*}
asymptotically almost surely. 
\end{proof}

Notice that
\[\frac{N\sigma^2}{\gamma_{\min}}\leq 1 \ll \sqrt{\log N}\frac{\gamma_{\max}}{\gamma_{\min}},\]
so using \eqref{eq:connectivity}, 
\begin{align}
\gamma_{\min}^{-1}\sqrt{N\sigma^2}&\ll\sqrt[4]{\frac{\log N}{\gamma_{\min}^{2}}\cdot\frac{\gamma_{\max}^2}{\gamma_{\min}^2}}\ll\sqrt[4]{\frac{\log N}{\gamma_{\min}}\cdot\frac{\gamma_{\max}^2}{\gamma_{\min}^2}}=o(1)\quad \text{and}\label{eq:otermnegligible}\\
\gamma_{\min}^{-1}(\log N\sqrt[4]{N\sigma^2})&\ll\sqrt[8]{\frac{\log^9 N}{\gamma_{\min}^{6}}\cdot\frac{\gamma_{\max}^2}{\gamma_{\min}^2}}\leq \sqrt[8]{\frac{\log^2 N}{\gamma_{\min}^2}\cdot\frac{\gamma_{\max}^4}{\gamma_{\min}^4}}\notag\\
&=\sqrt[4]{\frac{\log N}{\gamma_{\min}}\cdot\frac{\gamma_{\max}^2}{\gamma_{\min}^2}}=o(1)
\label{eq:otermnegligible2}\end{align} 
implying that the second $O$-term is always negligible compared to the first one, which by itself is a null sequence.

\begin{cor}\label{cor:allkappa}
Recall the definition of $\kappa$ \eqref{eq:kappa}.
	
	For $\kappa=0$ and $k=2,\dots,M$, 
	\[\lambda_k(B)\leq 1-{2}\cdot{\left(\frac{p_M}{(M-1)q}\right)^{-2}}(1+o(1))\]
	asymptotically almost surely.

	For $\kappa>0$ and $k=2,\dots,M$, there is a constant $c\in(0,2)$ such that
	\[\lambda_k(B)\leq 1-c(1+o(1))\]
	asymptotically almost surely.
\end{cor}
\begin{proof} 
Since $\lambda_m(B)\leq \lambda_2(B)$ for all $m\geq 2$, it is suffient to find a suitable bound for $\lambda_2(B)$.

In the case $\kappa=0$, we observe that $p_M\gg(M-1)q$. Therefore, 
\[{2}\cdot{\left(1+\frac{p_M}{(M-1)q}\right)^{-2}}={2}\cdot {\left(\frac{p_M}{(M-1)q}\right)^{-2}}(1+o(1)).\]
From \eqref{eq:qconditionneu}
we immediately find that 
\[{2}\cdot{\left(\frac{p_M}{(M-1)q}\right)^{-2}}\gg O\left(\sqrt[4]{\gamma_{\min}^{-1}\log N}\sqrt{\frac{\gamma_{\max}}{\gamma_{\min}}}\right).\]

	Applying Lemma \ref{lem:disassortativeP}, Proposition \ref{prop:eigenvaluesB} and \eqref{eq:otermnegligible}, the claim follows, as the $O$-term is negligible compared to the term ${2}{\left(\frac{p_M}{(M-1)q}\right)^{-2}}$ obtained from Lemma \ref{lem:disassortativeP}. 
	
In the case $\kappa>0$, we find that $\frac{p_M}{(M-1)q}$ is bounded by a constant for sufficiently large $N$. Therefore, ${2}{\left(1+\frac{p_M}{(M-1)q}\right)^{-2}}$ is bounded from below by a constant $c\in (0,2)$.
	Using again  Lemma \ref{lem:disassortativeP}, Proposition \ref{prop:eigenvaluesB} and \eqref{eq:otermnegligible} the claim follows. 
\end{proof}

\section{Spectral decomposition of the hitting time averaged over the target vertex and proof of Theorem \ref{maintheo}}
Using the spectral decomposition of the  hitting times according to \cite{Lov96} (c.f \cite{LoeweTorres,LT2020}) and Proposition \ref{prop:eigenvaluesB}, we obtain that with probability converging to 1 
\begin{align}
H^v&=\sum\limits_{k=2}^N\frac{1}{1-\lambda_k(B)}= \sum\limits_{k=2}^M\frac{1}{1-\lambda_k(B)}+\sum\limits_{k=M+1}^N\frac{1}{1-\lambda_k(B)}\label{eq:asympstartdecomp}
\end{align}
The second term is $(N-M)(1+o(1))$, i.e. $N(1+o(1))$ a.a.s.

In the case $\kappa>0$, we notice that $\frac{1}{1-\lambda_k(B)}=O(1)$ for $k=2, \ldots, M$ due to Corollary \ref{cor:allkappa}. Since $M\ll N$, 
$$
	\sum\limits_{k=2}^M\frac{1}{1-\lambda_k(B)}\leq  (M-1)\cdot O(1)=o(N),$$ 
so the first term on the right hand side 
in \eqref{eq:asympstartdecomp} is negligible compared to the second term.

In the case $\kappa=0$, due to \eqref{eq:connectivity} and \eqref{eq:qconditionneu} 
\[\frac{p_{M}^2}{2(M-1)q^2}\ll N\]
and thus 
by Corollary \ref{cor:allkappa}, we can bound
\begin{align}
\frac{1}{1-\lambda_k}&\leq 
\frac{p_M^2}{2((M-1)q)^2}(1+o(1))= o\left(\frac{N}{M-1}\right),\label{eq:inverseev}
\end{align}
for $k=2,\dots,M$ and thus
\begin{align*}
	\sum\limits_{k=2}^M\frac{1}{1-\lambda_k(B)}\leq  (M-1)\cdot o\left(\frac{N}{M-1}\right)=o(N),
\end{align*}
therefore again the first term on the right hand side 
in \eqref{eq:asympstartdecomp} is negligible compared to the second term.

\section{Spectral decomposition of the hitting time averaged over the starting vertex and proof of Theorem \ref{maintheo2}}

Using the spectral decomposition of the hitting times according to \cite{Lov96} (c.f \cite{LoeweTorres,LT2022}), we find
\begin{equation}\label{eq:avgtargetdecomp}
	H_w=\frac{2|E|-|L|}{d_w}\sum\limits_{k=2}^N\frac{1}{1-\lambda_k}u_{k,w}^2,
\end{equation}
where $u_{k,w}$ denotes the $w$-th entry of the eigenvector associated with the eigenvalue $\lambda_k=\lambda_k(B)$. We will always normalize
these eigenvectors to have length $1$.  Recall that $L$ denotes the set of loops in the graph (the proof of \cite{Lov96} can be modified accordingly).
We immediately notice using \eqref{eq:degreechernoff} that 
\begin{align}
	\frac{2|E|-|L|}{d_w}&=\frac{\sum\limits_{m=1}^M\sum\limits_{v\in V^{(m)}}d_v}{d_w}=\frac{\sum\limits_{m=1}^M\frac NM \gamma_m}{\gamma_{B(w)}}(1+o(1))=\frac NM\cdot \frac{\sum\limits_{m=1}^M \gamma_m}{\gamma_{B(w)}}\cdot(1+o(1)),\label{eq:fronttermasympt}
\end{align}
asymptotically almost surely (recall that $B(v)$ denotes the block which vertex $v\in V_N$ belongs to).

We order the eigenvalues and abbreviate the corresponding normalized eigenvectors of $B$ such that $\lambda_1\geq \dots\geq \lambda_N$ and $u_1,\dots,u_n$ (where $u_{k,v}$ denotes the $v$-th entry of the eigenvector $u_k$ corresponding to the eigenvalue $\lambda_k$).

We proceed in a similar way as in  \cite{LT2022}: Using that 
that $|\lambda_k|<1$ and we can therefore apply a geometric series, we note that 
\begin{align*}Z_N\coloneqq & \sum\limits_{k=2}^{N}\frac{1}{1-\lambda_k}u_{k,w}^2=\sum\limits_{k=2}^{N}\sum\limits_{m=0}^\infty\lambda_k^mu_{k,w}^2=\sum\limits_{k=2}^{N}\left(1+\lambda_k+\lambda_k^2\sum\limits_{m=0}^\infty\lambda_k^m\right)u_{k,w}^2\\
= &\sum\limits_{k=2}^{N}\left(1+\lambda_k+\lambda_k^2\frac{1}{1-\lambda_k}\right)u_{k,w}^2
\end{align*}
Now one immediately checks that the eigenvector corresponding to $\lambda_1$ is the componentwise square root of the stationary distribution $\pi$, i.e.\  
$u_{1,w}^2=\pi_w$. Moreover, the matrix of the eigenvectors $\mathcal{U}:=(u_{k,w})_{k,w=1}^N$ is unitary, which implies that  
$\sum\limits_{k=1}^{N}u_{k,w}^2=1$.
From the unitarity of $\mathcal{U}$ we obtain
\begin{align*}\sum\limits_{k=2}^{N}\lambda_ku_{k,w}^2=\sum\limits_{k=1}^{N}\lambda_ku_{k,w}^2-\pi_w&=\sum\limits_{k=1}^{N}\sum\limits_{v=1}^N B_{w,v}u_{k,v}u_{k,w}-\pi_w\\
	&=\sum\limits_{v=1}^N 
B_{w,v}\cdot \langle u_v,u_w\rangle-\pi_w=B_{w,w}-\pi_w.\end{align*}
Here we denote the entries of the matrix $B$ by $B_{v,w}$. 

Thus we arrive at
\begin{align}\label{eq:Zn}
	Z_N=&\sum\limits_{k=2}^{N}u_{k,w}^2+\sum\limits_{k=2}^{N}\lambda_ku_{k,w}^2
	+\sum\limits_{k=2}^{N}\frac{1}{1-\lambda_k}\lambda_k^2u_{k,w}^2=1+B_{w,w}-2\pi_w+\sum\limits_{k=2}^{N}\frac{\lambda_k^2}{1-\lambda_k}u_{k,w}^2.
\end{align}

We will bound the sum on the right hand side of \eqref{eq:Zn} in a separate lemma:

\begin{lem}
	Under the conditions of Theorem \ref{maintheo2},
	\[\sum\limits_{k=2}^N\frac{\lambda_k^2}{1-\lambda_k}u_{k,w}^2=o(1)\]
	with probability converging to 1. 
	\label{lem:spectraltermnegl}
\end{lem}

\begin{proof}	
	
	Using \eqref{eq:degreechernoff} for the inequality below we obtain
	\begin{align*}
	\sum\limits_{k=1}^N\lambda_k^2u_{k,w}^2&=\sum\limits_{k=1}^N(\lambda_ku_{k,w})^2=\sum\limits_{k=1}^N\Big(\sum\limits_{v=1}^NB_{w,v}u_{k,v}\Big)^2\\
	&=\sum\limits_{k=1}^N\sum\limits_{v=1}^N\sum\limits_{v'=1}^NB_{w,v}B_{w,v'}u_{k,v}u_{k,v'}=\sum\limits_{v=1}^N\sum\limits_{v'=1}^NB_{w,v}B_{w,v'}\sum\limits_{k=1}^Nu_{k,v}u_{k,v'}\\
	&=\sum\limits_{v=1}^N\sum\limits_{v'=1}^NB_{w,v}B_{w,v'}\langle u_{\boldsymbol{\cdot},v},u_{\boldsymbol{\cdot},v'}\rangle=\sum\limits_{v=1}^N\sum\limits_{v'=1}^NB_{w,v}B_{w,v'}\delta_{v,v'}=\sum\limits_{v=1}^NB_{w,v}^2=\sum\limits_{v=1}^N\frac{a_{v,w}}{d_vd_w}\\
	&\leq \frac{1}{\gamma_{\min}-\sqrt{\log N\cdot\gamma_{\min}}}\cdot\frac{1}{d_w}\cdot\sum\limits_{v=1}^Na_{v,w}=\frac{1}{\gamma_{\min}-\sqrt{\log N\cdot\gamma_{\min}}}
	\end{align*}
	with probability converging to 1.
	Therefore, we find that
	\begin{multline}
		\sum\limits_{k=2}^N\lambda_k^2u_{k,w}^2=\sum\limits_{k=1}^N\lambda_k^2u_{k,w}^2-u_{1,w}^2 \\ \leq \frac{1}{\gamma_{\min}-\sqrt{\log N\cdot\gamma_{\min}}}-\pi_w=\frac{1}{\gamma_{\min}}(1+o(1))-\frac{1}{N}(1+o(1))=O(\gamma_{\min}^{-1})\label{eq:znterm}
	\end{multline}
	with probability converging to 1.
	%
	%
	
	For $\kappa>0$, we apply Corollary \ref{cor:allkappa} and \eqref{eq:znterm} to obtain, with probability converging to 1,
	\begin{align*}
	\sum\limits_{k=2}^N\frac{\lambda_k^2}{1-\lambda_k}u_{k,w}^2\leq \frac{1+o(1)}{c}\cdot \sum\limits_{k=2}^N{\lambda_k^2}u_{k,w}^2\leq \frac{C}{\gamma_{\min}}
	\end{align*}
	for some suitable constant $C>0$. This converges to 0.
	
	For $\kappa=0$, 
	we use Corollary \ref{cor:allkappa} to obtain
	\begin{align} 
	\sum\limits_{k=2}^N\frac{\lambda_k^2}{1-\lambda_k}u_{k,w}^2\leq \frac{p_{\min}^2}{2((M-1)q)^2}\cdot \sum\limits_{k=2}^N{\lambda_k^2}u_{k,w}^2\leq \frac{p_{\min}^2}{2((M-1)q)^2}\cdot O(\gamma_{\min}^{-1}).
	\label{eq:znterm2}
	\end{align}
For $\kappa=0$ we have $\gamma_{\min}=\frac{N}{M}p_{\min}(1+o(1))$ such that
we just need to estimate $\frac{Mp_{\min}}{N(M-1)^2q^2}$.
However,
due to \eqref{eq:qconditionneu} and \eqref{eq:connectivity} we obtain
\begin{align*}
	\frac{p_{\min}^2}{2((M-1)q)^2\gamma_{\min}}
	&\ll \frac{p_{\min}^2}{\gamma_{\min}\left(\frac{M\log N}{N}\right)^{1/4}p_{\min}^{5/4}p_{\max}^{1/2}}
	\leq\frac{1}{\gamma_{\min}^{3/4}(1+o(1))(\log N)^{1/4}}\cdot \frac{p_{\min}^{1/2}}{p_{\max}^{1/2}}\leq1
\end{align*}
So the expression in \eqref{eq:znterm2} converges to 0.
\end{proof}

With this lemma at hand, we can now simplify $Z_N$. Starting from \eqref{eq:Zn} 
we obtain
\begin{align*}
	Z_N&=1+B_{w,w}-2\pi_w+\sum\limits_{k=2}^N\frac{\lambda_k^2}{1-\lambda_k}u_{k,w}^2=1+\frac{a_{w,w}}{d_w}-2\pi_w+o(1)=1+o(1)
\end{align*}
asymptotically almost surely, using \eqref{eq:degreechernoff} and \eqref{eq:fronttermasympt}.
Therefore, putting things together, from \eqref{eq:avgtargetdecomp} we see that
$$
H_w=\frac{N}{M}\frac{\sum\limits_{m=1}^M\gamma_m}{\gamma_{B(w)}}(1+o(1))(1+o(1))
$$
asymptotically almost surely, as proposed.

\section{A central limit theorem for $H_w$}

In order to obtain a central limit theorem for $H_w$, the same decomposition as previously seen in 
\eqref{eq:avgtargetdecomp} will become useful. Indeed, we will see that the scaling as well as the centering of the central limit 
theorem solely depend on the factor $\frac{2|E|}{d_w}$ again (see also the statement of Theorem \ref{thm:clttargethittingtimediffp}). This 
means that the statement in \eqref{eq:fronttermasympt} has to be enhanced to a convergence in distribution. Similarly, negligibility 
of the spectral term has to be made more precise as well to account for the scaling of the central limit theorem.

Consider the following definitions of the expected number and variance of edges within and between blocks, $\mu_{\mathrm{in}}$ and $\mu_{\mathrm{out}}$:
with
\begin{align*}
	\mu_m&\coloneqq \binom{N/M+1}{2}p_m, \quad m=1,\dots,M\\
	\mu_{\mathrm{in}}&\coloneqq \sum\limits_{m=1}^M\mu_m,\quad\\
	\mu_{\mathrm{out}}&\coloneqq\bigg[\binom{N+1}{2}-M\binom{N/M+1}{2}\bigg]q=\binom{M}{2}\left(\frac{N}{M}\right)^2q\\
	\tau_m^2&\coloneqq \binom{N/M+1}{2}p_m(1-p_m), \quad m=1,\dots,M\\
	\tau_{\mathrm{in}}^2&\coloneqq \sum\limits_{m=1}^M\tau_m^2,\quad\\
	\tau_{\mathrm{out}}^2&\coloneqq\bigg[\binom{N+1}{2}-M\binom{N/M+1}{2}\bigg]q(1-q)=\binom{M}{2}\left(\frac{N}{M}\right)^2q(1-q)\\
	\tau^2&\coloneqq\tau_{\mathrm{in}}^2+\tau_{\mathrm{out}}^2.
\end{align*}
Notice that $\E[|E|]=\mu_{\mathrm{in}}+\mu_{\mathrm{out}}$.
and hence, due to 
$2|E|=\sum\limits_{v\in V}d_v+|L|$,
\begin{equation} 2(\mu_{\mathrm{in}}+\mu_{\mathrm{out}})=\sum\limits_{m=1}^M\frac NM\gamma_m+\sum\limits_{m=1}^M\frac NM p_m=\frac NM\sum\limits_{m=1}^M(\gamma_m+p_m).
	\label{eq:edgesdegreesexpectation}
\end{equation}


The proof of Theorem \ref{thm:clttargethittingtimediffp} can be split into several parts. More precisely, we break down 
the spectral decomposition of $H_w$ into its factors.

Then, the following proposition holds true: 
\begin{prop}\label{prop:clt_domterm_diffp}
	Assume that the conditions of Theorem \ref{thm:clttargethittingtimediffp} hold, then
	\[	\frac{\gamma_m}{N\upsilon_m}\cdot\frac{\gamma_m}{\bar \gamma}\cdot\left(2\frac{|E|-\frac{|L|}{2}}{d_w}-\frac{N\bar\gamma}{\gamma_m}\right)\dconv \mathcal{N}(0,1).\]
\end{prop}
The proof of Proposition \ref{prop:clt_domterm_diffp} will be given with the help of several lemmas.

\begin{lem}\label{lem:diffpEconv}
	Under the conditions of Theorem \ref{thm:clttargethittingtimediffp} the following central limit theorem holds: 
	\begin{equation*}
		\frac{1}{{\tau}}\Big(|E|-(\mu_{\mathrm{in}}+\mu_{\mathrm{out}})\Big)\dconv \mathcal{N}(0,1).
	\end{equation*}
\end{lem}
Notice that condition \eqref{eq:spectralcondition_diffp} will not be necessary for the proof of this lemma.
\begin{proof}
	For two vertices $v,w\in V$ we denote by $\eps_{v,w}$ the random variable valued 1 if there is an edge between $v$ and $w$ and 0 otherwise. Clearly, $\eps_{v,w}$ is Bernoulli-distributed with parameter $q$ is $v$ and $w$ are in different blocks and with parameter $p_m$ if they are both contained in the same block $V^{(m)}$, $m=1,\dots,m$.

	Then 
	\[|E|=\sum\limits_{v,w\in V_N}\eps_{v,w}\]
	with
	\[\E |E| = \mu_{\mathrm{in}}+\mu_{\mathrm{out}}\quad\text{and}\quad \V{ |E|}=\tau_{\mathrm{in}}^2+\tau_{\mathrm{out}}^2=\tau^2\]
	For $v,w\in V_N$ let $\alpha_{v,w}=\eps_{v,w}-\E \eps_{v,w}$. Then $\sum\limits_{v,w\in V}\alpha_{v,w}$ can be interpreted as a scheme of independent centered random variables with existing variances  and $\sum\limits_{v,w\in V}\V{\alpha_{v,w}}=\tau^2<\infty$.
	
	Then notice that $\tau^2\to\infty$ as $N\to\infty$ due to \eqref{eq:connectivity}. Therefore, 
	\[\1_{\{|\alpha_{v,w}|>\eps\tau\}}= 0\] for sufficiently large $N$ and therefore,
	\[\EW{\alpha_{v,w}^2\1_{\{|\alpha_{v,w}|>\eps\tau\}}}=0\] for sufficiently large $N$, thereby yielding that Lindeberg's condition for the scheme of random variables $(\alpha_{v,w})_{v,w\in V}$ 
	\[\lim\limits_{N\to\infty}\frac{1}{\sum\limits_{v,w\in V}\V{\alpha_{v,w}}}\sum\limits_{v,w\in V}\EW{\alpha_{v,w}^2\1_{\{|\alpha_{v,w}|>\eps\tau\}}}=0\]
	is satisfied. Therefore,
	\[\frac{1}{\tau}(|E|-(\mu_{\mathrm{in}}+\mu_{\mathrm{out}}))=\frac{1}{\sum\limits_{v,w\in V}\V{\alpha_{v,w}}}\sum\limits_{v,w\in V}\alpha_{v,w}\dconv \mathcal N(0,1).\qedhere\]
\end{proof}

We are now ready to prove Proposition \ref{prop:clt_domterm_diffp}: 
\begin{proof}[Proof of Proposition \ref{prop:clt_domterm_diffp}]
	Assume that $w\in V^{(m)}$	Note that the degree $d_w$ can be written as $$d_w= \sum_{v\in V} \eps_{v,w}$$ which has expectation $\gamma_m$ and variance $\upsilon_m^2\coloneqq\frac{N}{M}p_m(1-p_m)+(M-1)\frac NMq(1-q)$
	Thus,	by the classical Lindeberg-Feller central limit theorem,
	\begin{equation}
		\frac{1}{\upsilon_m}(d_w-\gamma_m)\dconv \mathcal{N}(0,1)
		\label{eq:cltdw}
	\end{equation}
	
	and furthermore for the number of loops in the graph $|L|=\sum\limits_{v\in V}a_{v,v}$,
	\begin{align}\label{eq:loopclt_diffp}
		\frac{|L|-\sum\limits_{m=1}^M\frac{N}{M}p_m}{\sqrt{\sum\limits_{m=1}^M\frac{N}{M}p_m(1-p_m)}}\dconv\mathcal N(0,1)
	\end{align}
	
	We notice that 
	\begin{align}
		\hspace{0.5cm}&\hspace{-0.5cm}\frac{1}{\tau}\Big(|E|-\frac{|L|}{2}-\big(\mu_{\mathrm{in}}+\mu_{\mathrm{out}}-\frac12{\sum\limits_{m=1}^M\frac{N}{M}p_m}\big)\Big)
		=\frac{1}{\tau}\Big(|E|-\big(\mu_{\mathrm{in}}+\mu_{\mathrm{out}}\big)\Big)\notag\\
		&\qquad\qquad-\frac12
		\frac{\sqrt{\sum\limits_{m=1}^M\frac{N}{M}p_m(1-p_m)}}{\tau}\cdot\sqrt{\frac{1}{\sum\limits_{m=1}^M\frac{N}{M}p_m(1-p_m)}}\Big({|L|}-{\sum\limits_{m=1}^M\frac{N}{M}p_m}\Big)\label{eq:decompEminusL}
	\end{align}
	Here, $\tau=\sqrt{\tau_{\mathrm{in}}^2+\tau_{\mathrm{out}}^2}\geq \tau_{\mathrm{in}}=\sqrt{\sum\limits_{m=1}^M\binom{N/M+1}{2} p_m(1-p_m)}$ and thus 
	\[	\frac{\sqrt{\sum\limits_{m=1}^M\frac{N}{M}p_m(1-p_m)}}{\tau}\leq 	\frac{\sqrt{\sum\limits_{m=1}^M\frac{N}{M}p_m(1-p_m)}}{\sqrt{\sum\limits_{m=1}^M\binom{N/M+1}{2} p_m(1-p_m)}}=\sqrt{\frac{\frac NM}{\binom{N/M+1}{2} }}=\sqrt{\frac{2}{{N/M+1}}}\]
	converges to 0. Therefore, the expression in the second line of \eqref{eq:decompEminusL} converges to 0 in probability and by 
	Slutzky's theorem using Lemma \ref{lem:diffpEconv} and \eqref{eq:loopclt_diffp}, we find that there is a sequence of random variables $Z_N$, converging to a normal random variable $Z$  with mean 0 and variance $1$ such that
	\begin{equation}
		|E|-\frac{|L|}{2}=\mu_{\mathrm{in}}+\mu_{\mathrm{out}}-\frac12{\sum\limits_{m=1}^M\frac{N}{M}p_m}+\tau\cdot Z_N
		\label{eq:edgesconvergence}
	\end{equation}
	and from  \eqref{eq:cltdw} we analogous obtain that there is a sequence of random variables $Z_N'$ converging to a normal random variable $Z'$ with mean 0 and variance $1+\tilde\kappa$ such that 
	\begin{equation}
		d_w=\gamma_m+\upsilon_m\cdot Z_N'.
		\label{eq:degreeconvergence}
	\end{equation}
	
	Moreover,
	\begin{align}
		\frac{\gamma_m}{N\upsilon_m}&\cdot\frac{\gamma_m}{\bar \gamma}\cdot \frac{1}{2\gamma_m}\cdot\frac NM\sum\limits_{l=1}^Mp_l=\frac12\cdot\frac{\gamma_m}{\bar \gamma}\cdot\frac{\bar p}{\upsilon_m}\nconv 0\label{eq:clt_negl_term1}
	\end{align}
	due to \eqref{eq:cltadditionalassump} and	
	\begin{align}
		\frac{\gamma_m}{N\upsilon_m}&\cdot\frac{\gamma_m}{\bar \gamma}\cdot \frac{\tau}{\gamma_m}=\frac{\tau}{N\upsilon_m}\cdot\frac{\gamma_m}{\bar \gamma}\\
		&=\frac{1}{N\upsilon_m}\sqrt{\sum\limits_{l=1}^M\binom{N/M+1}{2}p_l(1-p_l)+\binom{M}{2}\left(\frac{N}{M}\right)^2q(1-q)}\cdot\frac{\gamma_m}{\bar \gamma}\notag\\
		&\approx\frac{\sqrt{\frac N2}}{N\upsilon_m}\cdot\frac{\gamma_m}{\bar \gamma}\sqrt{ \frac1M\sum\limits_{l=1}^M\big(\frac NMp_l(1-p_l)+(M-1)\frac NMq(1-q)\big)}\notag\\
		&=\sqrt{\frac{2}{N}}\cdot\frac{\bar\upsilon}{\upsilon_m}\cdot\frac{\gamma_m}{\bar \gamma}\nconv 0\label{eq:longtermnegl}
	\end{align}
	due to \eqref{eq:cltadditionalassump} and \eqref{eq:cltadditionalassump2}.
	Overall, we obtain from \eqref{eq:edgesconvergence} and \eqref{eq:degreeconvergence}, \eqref{eq:edgesdegreesexpectation}
	
	\begin{align}
		\frac{\gamma_m}{N\upsilon_m}&\cdot\frac{\gamma_m}{\bar \gamma}\cdot\left(2\frac{|E|-\frac{|L|}{2}}{d_w}-\frac{N\bar\gamma}{\gamma_m}\right)\notag\\
		&=\frac{\gamma_m}{N\upsilon_m}\cdot\frac{\gamma_m}{\bar \gamma}\cdot\left(2\frac{\mu_{\mathrm{in}}+\mu_{\mathrm{in}}-\frac12\frac NM\sum\limits_{l=1}^Mp_l+\tau\cdot Z_N}{\gamma_m+\upsilon_m \cdot Z_N'}-\frac{N\bar\gamma}{\gamma_m}\right)\notag\\
		&=\frac{\gamma_m}{N\upsilon_m}\cdot\frac{\gamma_m}{\bar \gamma}\cdot\left(2\frac{\frac12\frac NM\sum\limits_{l=1}^M(\gamma_l+p_l)-\frac12\frac NM\sum\limits_{l=1}^Mp_l+\tau\cdot Z_N}{\gamma_m+\upsilon_m \cdot Z_N'}-\frac{N\bar\gamma}{\gamma_m}\right)\notag\\
		&=\frac{\gamma_m}{N\upsilon_m}\cdot\frac{\gamma_m}{\bar \gamma}\cdot\left(\frac{\frac NM\sum\limits_{l=1}^M\gamma_l+2\tau \cdot Z_N}{\gamma_m+\upsilon_m \cdot Z_N'}-\frac{N\bar\gamma}{\gamma_m}\right)\notag\\
		&=\frac{\gamma_m}{N\upsilon_m}\cdot\frac{\gamma_m}{\bar \gamma}\cdot\left(\frac{\frac{N\bar\gamma}{\gamma_m}+2\frac{\tau}{\gamma_m} \cdot Z_N}{1+\frac{\upsilon_m}{\gamma_m} \cdot Z_N'}-\frac{N\bar\gamma}{\gamma_m}\right)\notag\\
		&=\frac{\gamma_m}{N\upsilon_m}\cdot\frac{\gamma_m}{\bar \gamma}\cdot\frac{\frac{N\bar\gamma}{\gamma_m}+2\frac{\tau}{\gamma_m} \cdot Z_N-\frac{N\bar\gamma}{\gamma_m}-\frac{N\bar\gamma}{\gamma_m}\frac{\upsilon_m}{\gamma_m}Z_N'}{1+\frac{\upsilon_m}{\gamma_m} \cdot Z_N'}\notag\\
		&=\frac{\gamma_m}{N\upsilon_m}\cdot\frac{\gamma_m}{\bar \gamma}\cdot\frac{2\frac{\tau}{\gamma_m} \cdot Z_N-\frac{N\bar\gamma}{\gamma_m}\frac{\upsilon_m}{\gamma_m}Z_N'}{1+\frac{\upsilon_m}{\gamma_m} \cdot Z_N'}\notag\\
		&=\frac{2\frac{\gamma_m}{N\upsilon_m}\cdot\frac{\gamma_m}{\bar \gamma}\cdot\frac{\tau}{\gamma_m} \cdot Z_N}{1+\frac{\upsilon_m}{\gamma_m} \cdot Z_N'}-\frac{\frac{\gamma_m}{N\upsilon_m}\cdot\frac{\gamma_m}{\bar \gamma}\cdot\frac{N\bar\gamma}{\gamma_m}\frac{\upsilon_m}{\gamma_m}Z_N'}{1+\frac{\upsilon_m}{\gamma_m} \cdot Z_N'}\notag\\
		&=\frac{2\frac{\gamma_m}{N\upsilon_m}\cdot\frac{\gamma_m}{\bar \gamma}\cdot\frac{\tau}{\gamma_m} \cdot Z_N}{1+\frac{\upsilon_m}{\gamma_m} \cdot Z_N'}-\frac{Z_N'}{1+\frac{\upsilon_m}{\gamma_m} \cdot Z_N'}\notag
	\end{align}
	In both terms the denominator converges to 1 in probability as $\upsilon\leq \sqrt{\gamma_m}$ follows from $1-p, 1-q\leq 1$ and thus $\frac{\upsilon_m}{\gamma_m}\leq\frac{1}{\sqrt{\gamma_m}}\to0$ as $N\to\infty$.
	
	From \eqref{eq:longtermnegl} we notice that the coefficient of $Z_N$ in the numerator of the first term converges to 0, therefore the entire first term converges in probability to 0 according to Slutzky's theorem.
	
	The second term converges to a standard normal distributed random variable according to Slutzky's theorem. This concludes the proof.
\end{proof}

A final ingredient is needed to complete the proof of the central limit theorem:

\begin{prop}\label{prop:secondpartnegligible_diffp}
	Under the assumptions of Theorem \ref{thm:clttargethittingtimediffp}
	\begin{equation*}
		\frac{\gamma_m}{\upsilon_m}\left(\sum\limits_{k=2}^N\frac{1}{1-\lambda_k}u_{k,w}^2-1\right)\Pconv 0.
	\end{equation*}
\end{prop}
\begin{proof}
	To prove the proposition, let us rewrite the term in the brackets as above.
	\begin{align*}Z_N\coloneqq & \sum\limits_{k=2}^{N}\frac{1}{1-\lambda_k}u_{k,w}^2=1+B_{w,w}-2\pi_w+\sum\limits_{k=2}^N\frac{\lambda_k^2}{1-\lambda_k}u_{k,w}^2.
	\end{align*}
	
	As $\pi_v=\frac{d_v}{2|E|}$ for all $v\in V$ and the $d_v$ are  identically distributed
	it is obvious that the $\pi_v$ are identically distributed for $v\in V$. 
	Furthermore,
	$\sum_{v\in V}d_v=2|E|$.
	Thus we obtain
	\[N\EW{\pi_w}=\sum\limits_{v\in V}\EW{\pi_v}=\frac{1}{2|E|}\E\Bigl[\sum\limits_{v\in V}d_v\Bigr]=1.\]
	Therefore,
	\[\E\Bigl[\frac{\gamma_m}{\upsilon_m}\pi_w\Bigr]=\frac{\gamma_m}{N\upsilon_m}\leq\frac{1}{\upsilon_m}\nconv 0\]
	and by non-negativity of $\frac{\gamma_m}{\upsilon_m}\pi_w$ and Markov's inequality,
	\begin{equation*}
		\frac{\gamma_m}{\upsilon_m}\pi_w\Pconv 0.
	\end{equation*}
	and similarly using \eqref{eq:degreechernoff} 
	\[
	\frac{\gamma_m}{\upsilon_m}B_{w,w}=\frac{\gamma_m}{\upsilon_m}\cdot\frac{a_{w,w}}{d_w}\Pconv 0.
	\]
	Finally, we have from  Lemma \ref{lem:spectraltermnegl} that
	\[\sum\limits_{k=2}^N\frac{\lambda_k^2}{1-\lambda_k}u_{k,w}^2\leq C\frac{p_{\min}^2}{\gamma_{\min}(M-1)^2q^2}\]
	Then due to \eqref{eq:spectralcondition_diffp}
	\[\frac{\gamma_m}{\upsilon_m}\sum\limits_{k=2}^N\frac{\lambda_k^2}{1-\lambda_k}u_{k,w}^2\nconv 0.\]
	So altogether, we obtain				 
	\[		 
	\frac{\gamma_m}{\upsilon_m}(Z_N-1)\Pconv 0.\qedhere\]
\end{proof}

\begin{proof}[Proof of Theorem \ref{thm:clttargethittingtimediffp}]
	We only give the proof in the case $\tilde\kappa\in(0,\infty)$, the other cases follow analogously.
	Notice that from Proposition \ref{prop:clt_domterm_diffp} and \ref{prop:secondpartnegligible_diffp} there are sequences of random variables $X_N$ converging to a standard normal random variable in distribution and $X_N'$ converging in probability to 0 such that
	\begin{align*}
		\frac{2|E|-|L|}{d_w}&=\frac{N\bar\gamma}{\gamma_m}+\frac{N\upsilon_m}{\gamma_m}\cdot\frac{\bar \gamma}{\gamma_m}\cdot X_N\quad\text{and}\\
		\sum\limits_{k=2}^N\frac{1}{1-\lambda_k}u_{k,w}^2&=1+\frac{\upsilon_m}{\gamma_m}\cdot X_N'
	\end{align*}
	Recalling the decomposition \eqref{eq:avgtargetdecomp},
	\begin{align*}
		\frac{\gamma_m}{N\upsilon_m}\cdot\frac{\gamma_m}{\bar \gamma}\cdot\left(H_w-\frac{N\bar\gamma}{\gamma_m}\right)\hspace{-4cm}&\hspace{4cm}=\frac{\gamma_m}{N\upsilon_m}\cdot\frac{\gamma_m}{\bar \gamma}\cdot\left(	\frac{2|E|-|L|}{d_w}\sum\limits_{k=2}^N\frac{1}{1-\lambda_k}u_{k,w}^2-\frac{N\bar\gamma}{\gamma_m}\right)\\
		&=\frac{\gamma_m}{N\upsilon_m}\cdot\frac{\gamma_m}{\bar \gamma}\cdot\left[\left(\frac{N\bar\gamma}{\gamma_m}+\frac{N\upsilon_m}{\gamma_m}\cdot\frac{\bar \gamma}{\gamma_m}\cdot X_N\right)\cdot\left(1+\frac{\upsilon_m}{\gamma_m} \cdot X_N'\right)-\frac{N\bar\gamma}{\gamma_m}\right]\\
		&=\frac{\gamma_m}{N\upsilon_m}\cdot\frac{\gamma_m}{\bar \gamma}\cdot\left[\frac{N\upsilon_m}{\gamma_m}\cdot\frac{\bar \gamma}{\gamma_m}\cdot X_N\cdot\left(1+\frac{\upsilon_m}{\gamma_m}\cdot X_N'\right)+\frac{N\upsilon_m}{\gamma_m}\cdot\frac{\bar \gamma}{\gamma_m} \cdot X_N'\right]\\
		&=X_N\cdot\left(1+\frac{\upsilon_m}{\gamma_m} \cdot X_N'\right)+ X_N'.
		\label{eq:CLTcomplete}
	\end{align*}
	Using the convergence in distribution of $X_N$ and the convergence in probability of $X_N'$, we immediately obtain convergence in distribution to a standard normal random variable.
\end{proof}

\section{The case of identical $p_m$}\label{sec:identicalpi}

If $p_1=\dots=p_M\eqqcolon p$ are identical, we can significantly improve on the results for general $p_m$. In this case, we are able to explicitly compute the eigenvalues of $P_M$ (and therefore $P_M'$): Let $\gamma = \frac{N}{M}(p+(M-1)q)$. 

In place of the previously established conditions on the $p_m$ and $q$, we now require

	\begin{equation}
		\frac{M \log^4(N)}{N p(N)+N(M-1)q(N)}\nconv 0. 
		\label{eq:connectivitypq}
	\end{equation}
	and 
	\begin{equation}
		\label{eq:qcondpq}
		q(N) \gg\sqrt{\frac{p(N)\log N}{NM}}.
	\end{equation}
	Finally, set	
	\begin{equation}
		\kappa\coloneqq \lim\limits_{N\to\infty}{\frac{(M-1)q}{p}}\in[0,\infty]\label{eq:tildekappa}
	\end{equation}
	
	As results we obtain
		\begin{thm}\label{maintheopq}
		Assume that conditions \eqref{eq:connectivitypq} and \eqref{eq:qcondpq} hold and that $\kappa$ is well defined. Then
		\[H^v = N(1+o(1))\]
		asymptotically almost surely.
	\end{thm}
	\begin{thm}\label{maintheo2pq}
		Assume that conditions \eqref{eq:connectivitypq} and \eqref{eq:qcondpq} hold and that $\kappa$ is well defined. Then
		\begin{align*}
			H_w=N(1+o(1))
		\end{align*}
		asymptotically almost surely.
	\end{thm}
	
	\begin{proof}[Proof of Theorems \ref{maintheopq} and \ref{maintheo2pq}]
	The proof is nearly identical to that for different $p_i$ with a few modifications

\begin{lem} In the case of $p_1=\dots=p_M=p$, the eigenvalues $\lambda_m(P_M), m=1,\dots,M$ of the matrix $P_M$ are given by $\lambda_1(P_M)=p+(M-1)q$ and $\lambda_m(P_M)=p-q$, $m\geq 2$. \label{lem:transitionprobeigenvalues}
	\end{lem}
	\begin{proof}
		This can be checked through direct computation.
	\end{proof}

After this, we can follow most of the proof from the previous sections. However, in place of Lemma \ref{lem:Reigenvalues}, we obtain

\begin{lem}\label{lem:Reigenvalues2}
		For $R$ as above, we have under the condition $\gamma\gg\log N$
		\[\|R\|_\infty\leq \sqrt{\frac{\log N}{\gamma}}\cdot(1+o(1))\]
		asymptotically almost surely.  
	\end{lem}
	\begin{proof} The proof is similar to that of Lemma 3.3 in \cite{LoeweTorres} but given here for the sake of completeness.
		Clearly,
		the entries of $R$ are given by $r_{v,w}=\frac{\gamma-\sqrt{d_vd_w}}{\gamma\sqrt{d_vd_w}}a_{v,w}$
		By setting $c={\sqrt{\log N}}$ in \cref{eq:degreechernoff}, 
		$$|d_v-\gamma|\leq  \sqrt{\log N\gamma} \quad \mbox{ for all $v=1,\dots,N$} 
		$$ 
		with probability converging to 1. Hence,
		\[|\gamma-\sqrt{d_vd_w}|\leq  \sqrt{\log N \gamma}\]
		with probability tending to 1.
		Furthermore,
		$$d_vd_w>(\gamma-\sqrt{\log N\gamma})^2$$ with probability converging to 1. Thus,
		\begin{align*}
			\|R\|_\infty=\max\limits_{v\in V}\sum\limits_{w\in V}|r_{v,w}|&=\max\limits_{v\in V}\sum\limits_{w\in V}\left|\frac{\gamma-
				\sqrt{d_vd_w}}{\gamma\sqrt{d_vd_w}}a_{v,w}\right|\\
			&\leq\frac{\sqrt{\log N \gamma}}{\gamma(\gamma-\sqrt{\log N\gamma})}\max\limits_{v\in V}\sum\limits_{w=1}^Na_{v,w}\\
			&= \frac{\sqrt{\log N \gamma}}{\gamma(\gamma-\sqrt{\log N\gamma})}\max\limits_{v\in V}d_v\\
			&\leq\frac{\sqrt{\log N \gamma}(\gamma+\sqrt{\log N\gamma})}{\gamma(\gamma-\sqrt{\log N\gamma})}\\
			&=\sqrt{\frac{\log N}{\gamma}}\cdot\frac{(\gamma+\sqrt{\log N\gamma})}{(\gamma-\sqrt{\log N\gamma})}\\
			&=\sqrt{\frac{\log N}{\gamma}}\cdot\frac{1+\sqrt{\frac{\log N}{\gamma}}}{1-\sqrt{\frac{\log N}{\gamma}}}\\
			&=\sqrt{\frac{\log N}{\gamma}}\cdot(1+o(1))
		\end{align*}
		since by assumption $\gamma \gg \log(N)$. 
	\end{proof}

We thus obtain from the explicit representation of the eigenvalues of $P_M$ 
\begin{align}
	\lambda_k(B)\leq 1-\frac{Nq}{\gamma}(1+o(1))\label{eq:eigenvalueexplicitsamep}
\end{align}
The remainder of the proof of the two theorems is then analogous to that for different $p_i$ but applying the previously improved results.
\end{proof}

The case of identical $p_i$ also allows a modified, improved version of the central limit theorem as given in Theorem \ref{thm:clttargethittingtimediffp}. The proof is nearly identical, with similar modifications as for the laws of large numbers.

This particular case furthermore allows to consider a central limit theorem for the hitting time averaged over the starting vertex. In particular, we can make the centering and scaling more precise and improve the necessary conditions as follows:

	\begin{thm}
		\label{thm:clttargethittingtime}
		Assume that conditions \eqref{eq:connectivitypq} -- \eqref{eq:tildekappa} are replaced by
			\begin{equation}
			\frac{M \log^4(N)}{N p(N) (1-p(N))+N(M-1)q(N)(1-q(N))}\nconv 0. 
			\label{eq:connectivitypq2}
		\end{equation}
		and 
		\begin{equation}
			\label{eq:qcondpq2}
			q(N) \gg\sqrt{\frac{p(N)\log N}{NM(1-p(N))}}.
		\end{equation}		
		and assume that	
		\begin{equation}
			\tilde\kappa\coloneqq \lim\limits_{N\to\infty}{\frac{(M-1)q(1-q)}{p(1-p)}}\in[0,\infty]\label{eq:tildekappa2}
		\end{equation} is well defined. Additionally, assume that the limit $\zeta\coloneqq \lim\limits_{N\to\infty}\frac{1-p}{1-q}$ is well defined and either a positive constant or (if $\tilde\kappa\in\{0,\infty\}$) equal to $\tilde\kappa$.
		Then
		\begin{align*}
			\rho_N\left(H_w-N\right)&\dconv \mathcal{N}(0,1-\alpha),
		\end{align*} 	
		where $\dconv $ denotes convergence in distribution and
		\[\rho_N\coloneqq\begin{cases}
			\sqrt{\frac{p}{NM(1-p)}},&\text{if $\tilde\kappa<\infty$}\\
			\sqrt{\frac{(M-1)q}{NM(1-q)}},&\text{if $\tilde\kappa=\infty$}
		\end{cases}\qquad\text{and}\qquad\alpha\coloneqq\begin{cases}
			\frac{\tilde\kappa\left(2\zeta-1+\tilde\kappa\zeta^2\right)}{(1+\zeta\tilde\kappa)^2},&\text{if $\tilde\kappa\in(0,\infty)$}\\
			0,&\text{if $\tilde\kappa\in\{0,\infty\}$}
		\end{cases}\]
	\end{thm}

Indeed, in the case when $p$ and $q$ satisfy the conditions of Theorem \ref{thm:clttargethittingtimediffp}, the result stated here is implied, which can be verified explicitly computing the scaling term in  Theorem \ref{thm:clttargethittingtimediffp} and comparing with the scaling terms given here for different values of $\tilde\kappa$ (as well as $\alpha$ in the case $\tilde\kappa\in(0,\infty)$). 
	
	The convergence in distribution of the respective dominating terms can be proven analogously to the proof of \ref{prop:clt_domterm_diffp}.
	The negligibility of the spectral term (Proposition \ref{prop:secondpartnegligible_diffp}) can be rewritten in the following way: 
	\begin{prop}\label{prop:secondpartnegligible}
		Under the assumptions of Theorem \ref{thm:clttargethittingtime}, for $\tilde\kappa<\infty$,
		\begin{equation*}\sqrt{\frac{Np}{M(1-p)}}\left(\sum\limits_{k=2}^N\frac{1}{1-\lambda_k}u_{k,w}^2-1\right)\Pconv 0.
		\end{equation*}
		Furthermore, for $\tilde\kappa=\infty$
		\begin{equation*}
			\sqrt{\frac{N(M-1)q}{M(1-q)}}\left(\sum\limits_{k=2}^N\frac{1}{1-\lambda_k}u_{k,w}^2-1\right)\Pconv 0.
		\end{equation*}
	\end{prop}
	\begin{proof}
		The proof follows analogously to the proof of Proposition \ref{prop:secondpartnegligible_diffp}. In the cases $\tilde\kappa>0$, this together with the definitions of $\tilde\kappa$ and $\zeta$ is sufficient.
		
		For $\kappa=0$, we notice that from \eqref{eq:eigenvalueexplicitsamep} together with similar computations to above,
		\[\sum\limits_{k=2}^N\frac{\lambda_k^2}{1-\lambda_k}u_{k,w}^2\leq \frac{C}{Nq}\] for a constant $C>0$ and additionally by \eqref{eq:qcondpq2},
		
		\[\sqrt{\frac{Np}{M(1-p)}}\frac{C}{Nq}\leq\frac{C\sqrt{p}}{\sqrt{NM(1-p)q^2}}\nconv 0.\qedhere\]
	\end{proof}
	
	The remainder of the proof of Theorem \ref{thm:clttargethittingtime} follows then analogously to that of Theorem \ref{thm:clttargethittingtimediffp}.

\bibliography{bibl}
\bibliographystyle{alpha}
\end{document}